\documentclass{amsart}

\usepackage{amssymb}
\usepackage{mathrsfs}
\usepackage{tikz}
\usepackage{soul}
\usepackage{xcolor}

\newtheorem{theorem}{Theorem}
\newtheorem{proposition}[theorem]{Proposition}
\newtheorem{lemma}[theorem]{Lemma}
\newtheorem{corollary}[theorem]{Corollary}

\newtheorem{remark}[theorem]{Remark}
\numberwithin{theorem}{section}

\theoremstyle{definition}
\newtheorem{definition}[theorem]{Definition}
\newtheorem{example}[theorem]{Example}

\begin{document}

\title[Racah problems for the oscillator algebra]{Racah problems for the oscillator algebra,\\the Lie algebra $\mathfrak{sl}_n$,\\ and multivariate Krawtchouk polynomials}
\author[N. Cramp\'e]{Nicolas Cramp\'e}
\address{Institut Denis-Poisson CNRS/UMR 7013 - Universit\'e de Tours - Universit\'e d'Orl\'eans, 
Parc de Grammont, 37200 Tours, France.}
\email{ crampe1977@gmail.com  }

\author[W. van de Vijver]{Wouter van de Vijver}
\address{Department of Electronics and Information Systems, Faculty of Engineering and Architecture, Ghent University, Building S8, Krijgslaan 281, 9000 Ghent, Belgium}
\email{Wouter.vandeVijver@UGent.be}

\author[L. Vinet]{Luc Vinet}
\address{Centre de Recherches Math\'ematiques, Universit\'e de Montr\'eal, P.O. Box 6128, Centre-ville Station, Montr\'eal, QC H3C 3J7, Canada}
\email{vinet@crm.umontreal.ca}

\date{\today}

\begin{abstract}
 The  oscillator Racah algebra $\mathcal{R}_n(\mathfrak{h})$ is realized by the intermediate Casimir operators arising in the multifold tensor product of the oscillator algebra $\mathfrak{h}$.   
 An embedding of the Lie algebra $\mathfrak{sl}_{n-1}$ into $\mathcal{R}_n(\mathfrak{h})$ is presented. 
It relates the representation theory of the two algebras. We establish the connection between recoupling coefficients for $\mathfrak{h}$ and 
matrix elements of $\mathfrak{sl}_n$-representations which are both expressed in terms of multivariate Krawtchouk polynomials of Griffiths type.
 
\end{abstract}

\maketitle

\tableofcontents

\section{Introduction}
This paper studies the oscillator Racah algebra  $\mathcal{R}_n(\mathfrak{h})$ viewed as the centralizer of the diagonal action of the oscillator 
algebra $\mathfrak{h}$ \cite{Streater} in the n-fold tensor product of its universal algebra. We shall find that it admits an embedding of $\mathfrak{sl}_{n-1}$. 
Building upon that result, we shall connect the facts that the multivariate Krawtchouk polynomials of Griffiths arise as $3(n-1)j$ symbols of $\mathfrak{h}$ 
as well as matrix elements of the restriction to the group $\textup{O}(n+1)$ of the symmetric representations of $\textup{SU}(n+1)$.

 There is growing interest in Racah algebras. These are, in particular, identified in the framework of Racah problems where one looks at the recouplings of 
 tensor products of certain Lie algebras. We shall denote by $n$ the number of factors. The cases with $n=3$ for the Lie algebra 
 $\mathfrak{su}(2)$ (or $\mathfrak{su}(1,1)$), the quantum algebra $U_q(\mathfrak{sl}_2)$ and the Lie superalgebra $\mathfrak{osp}(1|2)$ have first been examined. 
 They have led respectively to the (universal versions of the) Racah algebra $R(3)$ \cite{Granovskii&Zhedanov-1988,GVZ,Genest&Vinet&Zhedanov-2014}
 the Askey-Wilson algebra $AW(3)$ \cite{GZ,H-WH} and the Bannai-Ito algebra $BI(3)$ \cite{GVZ2}. In this picture, where there is an implicit map from the abstract
 Racah algebra onto the centralizer of the diagonal action of say, $\mathfrak{su}(2)$, $U_q(\mathfrak{sl}_2)$ or $\mathfrak{osp}(1|2)$ on their triple product, the images 
 of the three generators of the Racah algebra are expressed in terms of the intermediate Casimir elements. The representations of these algebras encompass the bispectral 
 properties of the orthogonal polynomials bearing the same name that are essentially the Racah  or 6j- coefficients of the corresponding algebras whose triple tensor products 
 are considered. In fact this is how the $AW(3)$ was first identified \cite{Z91} through its realization in terms of the recurrence and q-difference operators of the Askey-Wilson polynomials. 
 These Racah algebras have arisen in numerous contexts. They have appeared as symmetry algebras of superintegrable models \cite{GVZ3,DGTVZ}, are featuring centrally in aspects of algebraic
 combinatorics \cite{TZ} and are related to the Leonard pairs \cite{TV}. They have been related to Double Affine Hecke Algebras (DAHA) \cite{Koorn1,Koorn2,T2} and degenerate cases \cite{GVZ4}. 
 Algebras over three strands such as the Temperley-Lieb  or Brauer ones that arise in Schur-Weyl duality have been shown to be quotients of Racah algebras \cite{CPV,CFV}. 
 Isomorphisms with certain Kauffman-Skein algebras have been established \cite{BP,C}. Howe duality could be used to relate different presentations \cite{GVVZ,GVVZ2,FGRV}. 
 Truncated reflection algebras attached to $U_q(\hat{\mathfrak{sl}}_2)$, to loop algebra or to the Yangian of $sl_2$ have also been found \cite{Bas,BasC,CRVZ} to lead to $AW(3)$ or Racah algebras. 
 Finally, the description of the Bannai-Ito and Askey-Wilson algebras 
 has been cast recently in the framework of the universal $R$-matrix \cite{CVZ1,CVZ2}. This offers sufficient cause already to warrant the exploration of the Racah algebra associated to the oscillator algebra.
 
 The study of Racah algebras as centralizers of $n$-fold tensor products with $n$ larger than $3$ has been pursued \cite{DeBie&Genest$vandeVijver&Vinet,DGV,DV,PW}. 
 The recoupling coefficients in these instances are orthogonal polynomials in many variables. In the case of the generalized Racah algebra for example, bases for representations are obtained 
 by diagonalizing the generators of different maximal Abelian subalgebras \cite{DeBie&Genest$vandeVijver&Vinet} and the connection coefficients between two such bases are given in terms of 
 multivariate Racah polynomials of the Tratnik type \cite{Tratnik-1991,GI}. Given that the Racah polynomials sit at the top of the finite part of the 
 q=1 Askey scheme \cite{Koekoek,Koekoek&Lesky&Swarttouw-2010}, these Tratnik polynomials provide, through specializations and limits, multivariable extensions of all the finite 
 families of orthogonal polynomials in parallel with what occurs in the univariate situation.
 
 We shall here examine the oscillator Racah algebra $\mathcal{R}_n(\mathfrak{h})$ for arbitrary $n$ thereby exploring the structure that encodes the properties of the $3nj$-symbols of the 
 oscillator algebra $\mathfrak{h}$. As will be seen, these are given in terms of multivariate Krawtchouk polynomials. Historically, the $3j$- and $6j$- coefficients were 
 obtained in \cite{Vanderjeugt,Klimyk&Vilenkin} and found to be given both in terms of univariate Krawtchouk polynomials. Looking at the $9j$- symbols of $\mathfrak{h}$, Zhedanov 
 made the observation  \cite{Zhedanov} that these involve polynomials in two variables orthogonal with respect to the trinomial distribution and depending on one more parameter than the 
 Tratnik ones. We shall indicate below how this generalizes.
 
 Regarding multivariate Krawtchouk polynomials, it is worth recalling and clarifying the following points. Polynomials in $n$ variables that are orthogonal with respect to the multinomial 
 distribution were introduced by Griffiths in 1971 using a generating function \cite{Griffiths-1971}; these polynomials involve $\frac{1}{2}n(n-1)$ parameters. For a review see \cite{DG}.
 The specialization to the Krawtchouk family of the Racah polynomials in $n$ variables introduced by Tratnik in 1991 and mentioned before, yields polynomials also orthogonal with respect 
 to the multinomial distribution but depending in this case on $n$ parameters only (in addition to the maximal degree $N$) \cite{GI}. These two sets are hence not the same and their 
 relation remained unclear for some time largely because of their intricate parametrizations. The bivariate Krawtchouk polynomials of Griffiths were rediscovered in 2008 in connection 
 with a probabilistic model and as limits of the $9j$- symbols of $\mathfrak{su}(2)$ \cite{HR}; they were called Rahman polynomials for a while. Slightly before, Mizukawa and Tanaka \cite{MT} 
 had related the Griffiths polynomials to character algebras and provided an explicit formula in terms of Gel'fand-Aomoto hypergeometric series.
 
 Of special relevance to the present article is the group theoretical interpretation of the multivariable Krawtchouk polynomials of Griffiths that was given by Genest, Vinet and Zhedanov 
 in \cite{GVZ5} where they observe that these polynomials arise in the matrix elements of the representations of the orthogonal group $\textup{O}(n+1)$ that act on the energy eigenspaces of the 
 isotropic $(n+1)-$dimensional harmonic oscillator. In other words, they have shown that the matrix elements of the restriction to $\textup{O}(n+1)$ of the symmetric representations 
 of $\textup{SU}(n+1)$ are expressed in terms of the Krawtchouk polynomials of Griffiths; the parameters of the polynomials are thus interpreted as the $\frac{1}{2}n(n-1)$ parameters, 
 for instance the Euler angles, that specify rotations in $(n+1)$ dimensions. This cogent picture has allowed for a complete characterization on algebraic grounds of the Griffiths polynomials 
 (recurrence relations, difference equations, generating function etc.) using the covariance properties of the oscillator creation and annihilation operators under $\textup{O}(n+1)$. Furthermore 
 this approach clarified the connection between the Griffiths and Tratnik classes of Krawtchouk polynomials by making explicit that the latter is simply a special case of the former. 
 For example, in the bivariate case ($n=3$), while the Griffiths polynomials with $3$ parameters correspond to a general rotation in three dimensions, the Tratnik ones with $2$ parameters, 
 arise from rotations that are only products of two planar rotations about perpendicular axes. Related to this group theoretical interpretation is the work of Iliev and 
 Terwilliger \cite{Iliev&Terwilliger,Iliev-2012} (see also \cite{R}) where the Krawtchouk polynomials appear as overlap coefficients between basis elements for two 
 modules of $\mathfrak{sl}_{n+1}(\mathbb C)$ with the basis elements for the representation spaces defined as eigenvectors of two Cartan subalgebras related by an anti-automorphism specified by the parameters.
 
 The embedding of $\mathfrak{sl}_{n-1}$ into $\mathcal{R}_n(\mathfrak{h})$ that we shall construct will provide, besides its intrinsic algebraic interest, a connection between 
 these two manifestations of the multivariate Krawtchouk polynomials of Griffiths in matrix elements of representations and in recoupling coefficients. \\

The paper is organized as follows. In section \ref{sectiondef} we introduce the Racah algebra for the oscillator algebra. We also exhibit some properties and 
find a number of commutation relations that are needed to prove the main theorem of this paper given in the following section. 
In section \ref{mainsection} we show how to embed the special linear Lie algebra $\mathfrak{sl}_{n-1}$ into the Racah algebra for the oscillator algebra. We then 
study Abelian subalgebras of the Racah algebra for the oscillator algebra related to Cartan algebras of the special linear algebra. These are called labelling 
Abelian algebras and will be the main tool for section \ref{representationtheory}. In this section we connect the representation theories of $\mathfrak{sl}_{n-1}$ and of the Racah algebra for 
the oscillator algebra. We show how multivariate Krawtchouk polynomials both of Tratnik type and of Griffiths type appear as overlap 
oefficients between bases of irreducible representations diagonalized by the labelling Abelian algebras. We focus briefly on the relation with the $6j$- and $9j$-symbols. 
We finish this section by constructing for a number of overlaps an isomorphism of $\mathfrak{sl}_{n-1}$ and the corresponding rotation matrix. 
We explain the link of these rotation matrices with the multivariate Krawtchouk polynomials of Griffiths type.  A brief conclusion follows. Appendix \ref{Overlap} records for reference how the overlaps between representation eigenbases associated to equivalent $\mathfrak{sl}_2$ Cartan generators are obtained in terms of univariate Krawtchouk polynomials.

\section{The oscillator algebra $\mathfrak{h}$ and the Racah algebra}\label{sectiondef}

The oscillator algebra $\mathfrak{h}$ is the Lie algebra generated by four elements $A_{\pm}$, $A_0$ and a central element $a$ with following defining relations:
\begin{equation}\label{Oscillator}
[A_-,A_+]=a, \qquad [A_0, A_\pm]=\pm A_\pm.
\end{equation}
The Casimir element $Q$ is contained in the universal enveloping algebra $\mathcal{U}(\mathfrak{h})$ and is given by:
\begin{equation}\label{Casimir}
	Q:=aA_0-A_+A_-.
\end{equation}
We define the elements of $\mathcal{U}(\mathfrak{h})^{\otimes n}$ for $1\leq k  \leq n$
\begin{align}
A_{0,k}&:= 1^{\otimes (k-1)} \otimes A_0 \otimes  1^{\otimes (n-k)},\quad \\
A_{\pm,k}&:=1^{\otimes (k-1)} \otimes A_\pm \otimes  1^{\otimes (n-k)},\quad  \\
a_k&:=1^{\otimes (k-1)} \otimes a \otimes  1^{\otimes (n-k)}\label{eq:ai}.
\end{align}
and for any subset non-empty $K \subset [n]:=\{1,\ldots, n \}$
\begin{equation}
A_{0,K}:=\sum_{k\in K} A_{0,k}, \quad A_{\pm,K}:=\sum_{k\in K} A_{\pm,k}, \quad a_K:=\sum_{k\in K} a_{k}.\label{aK}
\end{equation}
We denote the Lie algebra (isomorphic to $\mathfrak{h}$) generated by the operators $A_{0,K}$, $A_{\pm,K}$ and $a_K$ by $\mathfrak{h}_K$. The Casimir element of this algebra is $Q_K$:
\begin{align}
\label{Casimir-Upper}
Q_K:=a_KA_{0,K}-A_{+,K}A_{-,K}.
\end{align}
 The operators $Q_K$ will define the algebra of interest of this article. 
\begin{definition}
We define the oscillator Racah algebra $\mathcal{R}_n(\mathfrak{h})$ to be the subalgebra of $\mathcal{U}(\mathfrak{h})^{\otimes n}$  generated by the elements of the set $\{ Q_K\,|\, K\subset [n] \text{ and }  K \neq \emptyset \}$.
\end{definition}
\begin{example}
The easiest non-trivial example is given for $n=3$. Then a generator is constructed for every non-empty $K \in [3]=\{1,2,3\}$. The set of generators are given by
\[
\{ Q_{1},Q_{2},Q_{3},Q_{12},Q_{13},Q_{23},Q_{123}\}.
\]
\end{example}
\begin{proposition}
The following relations hold in $\mathcal{U}(\mathfrak{h})^{\otimes n}$, for any non-empty $K \subset [n]$:
\[
[ Q_K, A_{ 0, [n] } ]=0, \qquad  [ Q_K, A_{ \pm, [n] } ]=0, \qquad [ Q_K, a_{ [n] } ]=0.
\]
Then $\mathcal{R}_n(\mathfrak{h})$ generates a subalgebra of the centralizer of the oscillator algebra $\mathfrak{h}_{[n]}$ in  $\mathcal{U}(\mathfrak{h})^{\otimes n}$.
\end{proposition}
We wish to find the defining commutation relations obeyed by the generators $Q_K$ of $\mathcal{R}_n(\mathfrak{h})$.  First we want to point out this lemma:
\begin{lemma}\label{Racah1}
Let $\{K_p\}_{p=1\ldots k}$ be a set of $k$ disjoint subsets of $[n]$. Define $K_B:=\cup_{q \in B} K_q$ with $B\subset [k]$. Consider the following map
\[  \theta: \mathcal{R}_k(\mathfrak{h}) \rightarrow \mathcal{R}_n(\mathfrak{h}) :Q_{B} \mapsto Q_{K_B}. \]
This is an injective morphism. We denote its image by $\mathcal{R}_k^{K_1,\ldots, K_k}(\mathfrak{h})$. This algebra is isomorphic to $\mathcal{R}_k(\mathfrak{h})$.
\end{lemma}

\begin{example}
Consider a partition of the set $\{1,2,3,4 \}$. For example take $K_1=\{1\}$, $K_2=\{2,4\}$ and $K_3=\{3\}$. 
Then we have the following injective morphism of $\mathcal{R}_3(\mathfrak{h})$ into $\mathcal{R}_4(\mathfrak{h})$:
\begin{align*}
 &\theta(Q_{1})= Q_{K_1}=Q_{1},& \theta(Q_{12})&=Q_{K_1K_2}=Q_{124},\\
 &\theta(Q_{2})= Q_{K_2}=Q_{24},& \theta(Q_{13})&=Q_{K_1K_3}=Q_{13},\\
 &\theta(Q_{3})= Q_{K_3}=Q_3,& \theta(Q_{23})&=Q_{K_2K_3}=Q_{234}, \\
 &\theta(Q_{123})=Q_{K_1K_2K_3}=Q_{1234}. & &
\end{align*}
Here we introduced the shortened notation $KL:=K \cup L $ for sets $K$ and $L$.
\end{example}

\begin{proof}
%We employ three morphisms. First define 
%\[
%\mu^*_i=\underbrace{1\otimes\ldots\otimes 1}_{i-1 \text{ times }} \otimes \mu^* \underbrace{1\otimes\ldots\otimes 1}_{k-i\text{ times }}.
%\]
%The map $\mu^*_i$ is a morphism from $\mathcal{R}_{k}(\mathfrak{h})$ into $\mathcal{R}_{k+1}(\mathfrak{h})$. The second morphism we use is $\tau_i$ as defined in formula \ref{tau} and thirdly the isomorphism $\pi_{ij}$. The map $\pi_{ij}$ flips the elements of the homogenous tensor products on position $i$ and $j$
%\begin{align*}
%\pi_{ij}(& t_1\otimes \dots \otimes \ t_{i-1} \otimes t_i \otimes t_{i+1} \otimes \dots \otimes t_{j-1} \otimes t_j \otimes  \otimes t_{j+1}\otimes \dots \otimes t_{k})= \\
%	&t_1\otimes \dots \otimes \ t_{i-1} \otimes t_j \otimes t_{i+1} \otimes \dots \otimes t_{j-1} \otimes t_i \otimes  \otimes t_{j+1} \otimes\dots \otimes t_{k}
%\end{align*}
%With these injective morphisms we construct the injective morphism$\tau$. We need to replace $i$ by $A_i$. To do so we first replace $i$ by $|A_i$ indices. We start with index $1$. Act with $\mu_1^{|A_1|-1}$ to replace $1$ by $[|A_i|]$. Next we do the same for index $2$ which after the action of $\mu_1^{|A_1|-1}$ lies on position $|A_1|+1$. We act with $\mu_{|A_1|+1}^{|A_2|-1}$. In this way the index $2$ is replaced by the set $[|A_1|+1 \dots |A_2| ]$. We continu this process

We repeat the strategy in \cite[section 4.2]{DeBie&Genest$vandeVijver&Vinet} and generalize to the $n$-fold tensor product space. In formulas (\ref{aK}) 
we constructed an algebra $\mathfrak{h}_K$ isomorphic to $\mathfrak{h}$ acting on the components of tensor product whose indices are in $K$.
%To map $\mathcal{R}_k(\mathfrak{h})$ into  $\mathcal{R}_n(\mathfrak{h})$ take $k$ disjoint subsets of $[n]$: $\{ K_p \subset [n] \,| \, p \in [k] \text{ and } K_i \cap K_j=\emptyset, \, \,i\neq j  \}$. 
Consider the algebra generated by the union of $\mathfrak{h}_{K_p}$.
\[
  \mathcal{U}(\mathfrak{h})^{\otimes k} \cong \langle \, \mathfrak{h}_{K_p}  \rangle_{p \in [k]}.
\]
The isomorphism is defined on the generators by $ 1^{\otimes( p-1)} \otimes X \otimes 1^{\otimes (k-p)} \rightarrow X_{K_p}$ 
where $X$ is one of the generators $A_\pm$, $A_0$ or $a$. Inside the algebra $\mathcal{U}( \mathfrak{h})^{\otimes k}$ we find $\mathcal{R}_k(\mathfrak{h})$ generated 
by the operators $Q_B:=Q_{i_1 i_2 \dots i_l}$ with $B \subset [k]$. Their images under this isomorphism are $Q_{K_B}:=Q_{K_{i_1} K_{i_2} \dots K_{i_l}}$. Hence the operators $Q_{K_B}$ with $B \subset [k]$ 
generate an algebra isomorphic to $\mathcal{R}_k(\mathfrak{h})$ inside $\mathcal{R}_n(\mathfrak{h})$.
\end{proof}
%\begin{lemma}\label{Racah1}
%Let $K$, $L$ and $M$ be three disjoint subsets of $[n]$. The algebra generated by the set
%\[ \{ Q_K,Q_L,Q_M,Q_{K \cup L}, Q_{K \cup M}, Q_{ L \cup M}, Q_{K \cup L \cup M}  \} \]
%is isomorphic to the rank $1$ algebra $\mathcal{R}_3(\mathfrak{h})$ .
%We denote this algebra by $\mathcal{R}_3^{K,L,M}(\mathfrak{h})$.
%\end{lemma}
Using the strategy of this proof, we can always replace indices by sets in any relation given. For example, consider the following relation:
\[
 Q_{123}=Q_{12}+Q_{13}+Q_{23}-Q_1-Q_2-Q_3.
\]
This can be found by straightforward calculation. Using Lemma \ref{Racah1} we automatically have for three disjoint sets $K$, $L$ and $M$ 
\begin{equation}\label{SumQ}
 Q_{KLM}=Q_{KL}+Q_{KM}+Q_{LM}-Q_K-Q_L-Q_M.
\end{equation}
This gives us a number of linear dependencies between the generators $Q_A$. From these dependencies one can prove the following:
\begin{lemma}\label{LinearDependence} For any set $K \subset [n]$, it holds that
\begin{equation*}
 Q_K=\sum_{\left\{i,j\right\}\subset K} Q_{ij}-\left(|K|-2\right)\sum_{i \in K} Q_i .
\end{equation*}
\end{lemma}
By Lemma \ref{LinearDependence} it suffices to find the commutation relations of the generating set $\{ Q_{ij}\}$. 

We also have the following lemma:
\begin{lemma}\label{commutator}
 If either $K \subset L$ or $L \subset K$ or $K \cap L=\emptyset$ then $Q_K$ and $Q_L$ commute.
\end{lemma}
\begin{proof}
 By construction we know that $[Q_1,Q_2]=0$ and $[Q_{1},Q_{12}]=0$. Replacing the indices by sets by Lemma \ref{Racah1} concludes the proof.
\end{proof}
In particular the elements $Q_{[n]}$ and $Q_{i}$ are central in $\mathcal{R}_n(\mathfrak{h})$.
This also means that $[Q_{ij},Q_{lk}]=0$ if $\{i,j\}=\{l,k\}$ or $\{i,j\}\cap\{l,k\}=\emptyset$. The operators $Q_{ij}$ and $Q_{lk}$ do not commute only if they have exactly one index in common. 
Investigation by computer shows that it is not possible to write $[Q_{ij},Q_{jk}]$ as a linear combination of the generators $Q_K$. 

The set of commutators $\{[Q_{ij},Q_{jk}]\}$ is also not linear independent. First we have by Lemma \ref{commutator}
\begin{align}\label{Fijk}
\begin{split}
0&=[Q_{ij},Q_{ijk}] \\
 &=[Q_{ij},Q_{ij}+Q_{ik}+Q_{jk}-Q_{i}-Q_{j}-Q_{k}] \\
&=[Q_{ij},Q_{jk}]+[Q_{ij},Q_{ik}].
\end{split}
\end{align}
In the second line we used formula (\ref{SumQ}). We conclude that $[Q_{ij},Q_{jk}]=[Q_{ik},Q_{ij}]$.  Similarly on can show that  $[Q_{jk},Q_{ij}]=[Q_{ik},Q_{jk}]$ by considering $[Q_{jk},Q_{ijk}]$ or, equivalently, switching the indices $i \leftrightarrow k$. We conclude for all $i$, $j$, $k$ in $[n]$:
\begin{equation}\label{Omega}
 [Q_{ij},Q_{jk}]=[Q_{jk},Q_{ik}]=[Q_{ik},Q_{ij}].
\end{equation}

Remember that the elements $a_i$ for $i=1 \dots n$ as defined in formula \eqref{eq:ai} are $n$ different central operators in the algebra $\mathcal{U}(\mathfrak{h})^{\otimes n}$. A tedious computation shows that the following linear relation also holds:
\begin{equation}\label{Commsum}
 a_i[Q_{jk},Q_{kl}]=a_j[Q_{ik},Q_{kl}]-a_k[Q_{ij},Q_{jl}]+a_l[Q_{ij},Q_{jk}].
\end{equation}

All double commutators are obtained from the following two expressions by switching indices, from the properties of the commutator and from relation \ref{Omega}:
\begin{align}
\begin{split} \label{relation1}
[[Q_{ij},Q_{jk}],Q_{ij}]&=a_k(a_i-a_j)Q_{ij}-a_i(a_i+a_j)Q_{jk}+a_j(a_i+a_j)Q_{ik} \\
				& -(a_j+a_k)(a_i+a_j)Q_i+(a_i+a_k)(a_i+a_j)Q_j+(a_i^2-a_j^2)Q_k ,
\end{split} \\
\begin{split} \label{relation2}
[[Q_{ij},Q_{jk}],Q_{kl}]&=a_ia_l(Q_{jk}-Q_j-Q_k)-a_ja_l(Q_{ik}-Q_i-Q_k) \\
				&-a_ia_k(Q_{jl}-Q_j-Q_l)+a_ja_k(Q_{il}-Q_i-Q_l).
\end{split} 
\end{align}
It follows that the generators also satisfy the following relations for all $i$, $j$ and $k$ in $[n]$:
\begin{equation}\label{Dolan-Grady}
 [[[Q_{ij},[Q_{ij},[Q_{ij},Q_{jk}]]]=(a_i+a_j)^2[Q_{ij},Q_{jk}] .
 \end{equation}
 Observe that relation (\ref{Dolan-Grady}) is the Dolan-Grady relation up to central elements as defined in \cite{Dolan&Grady}.

  Let us generate an algebra from the set $\{ Q_K\,|\, \emptyset \neq K \subset [n]\}$ using the Lie bracket instead of the ordinary multiplication on $ \mathcal{U}(\mathfrak{h})^{\otimes n}$. 
  Denote it by $\mathcal{R}_n(\mathfrak{h},[])$. From the relations above we conclude that $\mathcal{R}_n(\mathfrak{h},[])$ is generated as a vector space by 
  the operators $Q_K$ and their commutators $[Q_K,Q_L]$ over the field $\mathbb{R}(a_1,\dots, a_n)$. We perform some reductions. First the $Q_i$ and $Q_{ij}$ generate all $Q_K$ by Lemma \ref{LinearDependence}. 
  The commutators $[Q_{ij},Q_{jk}]$ can be written as a linear combination of $[Q_{1j},Q_{jk}]$ by formula \eqref{Commsum}. By formula \eqref{Fijk} 
  we have $[Q_{1j},Q_{jk}]=[Q_{1k},Q_{kj}]$ so we require that $j<k$. It follows then that as a vector space $\mathcal{R}_n(\mathfrak{h},[])$ is generated by the following set:
\begin{equation} \label{engendre}
  \{ Q_i\, | \,i \in [n] \} \cup \{ Q_{ij}\,| \, 1\leq i<j\leq n\} \cup \{ [Q_{1j}, Q_{jk}]\,| \,1 <j<k \leq n \}.
\end{equation}
We will prove later on that this is a basis for $\mathcal{R}_n(\mathfrak{h},[])$ as a vector space over the field $\mathbb{R}(a_1,\dots, a_n)$. Moreover, we will prove that the set of equalities \eqref{relation1}, \eqref{relation2} and \eqref{Commsum} together with Lemma \ref{commutator} and Lemma \ref{LinearDependence} exhausts all commutation relations of $\mathcal{R}_n(\mathfrak{h},[])$.

\section{Embedding of $\mathfrak{sl}_{n-1}$ into $\mathcal{R}_n(\mathfrak{h})$} \label{mainsection}

In this section we study the relationship between the special linear Lie algebra $\mathfrak{sl}_{n-1}$ and $\mathcal{R}_n(\mathfrak{h})$. The Lie algebra $\mathfrak{sl}_{n-1}$ is generated by the following set of elements: $\{ e_{kk+1},\, e_{k+1k}, \, h_k\, |\, 1\leq k \leq n-2\}$.  By applying Serre's theorem on $\mathfrak{sl}_{n-1}$ we are guaranteed of a full set of relations which we give here. To this end we introduce the Cartan matrix:
\begin{align*}
 A_{ij}=\begin{cases} 2 &\text{ if } i=j \\ -1 &\text{ if } j=i\pm 1 \\ 0 & \text{ if } |j-i|>1 \end{cases}.
\end{align*}
Here are the Chevalley-Serre relations:
\begin{align}
&[h_i,h_j]=0,& \label{S1} \\
&[e_{ii+1},e_{j+1j} ]=\delta_{ij}h_i, & \label{S2} \\
& [h_i,e_{jj+1}]=A_{ij}e_{jj+1}, \quad [h_i,e_{j+1j}]=-A_{ij}e_{j+1j}, & \label{S3}  \\
& \textup{ad}(e_{ii+1})^{1-A_{ij}}(e_{jj+1})=0,&\text{ if } i\neq j  \label{S+}\\
& \textup{ad}( e_{i+1i})^{1-A_{ij}}(e_{j+1j})=0. & \text{ if } i\neq j \label{S-}
\end{align}
The operator $\textup{ad}$ is the adjoint action: $\textup{ad}(x)(y):=[x,y]$. 
The set $\{ h_k\, |\, 1 \leq k \leq n-2\}$ generates the Cartan algebra of $\mathfrak{sl}_{n-1}$. When we consider the $\mathfrak{sl}_2$ case, there are three generators $\{e_{12},e_{21}, h_1\}$ with following relations:
\begin{equation}\label{relsl2}
 [e_{12},e_{21}]=h_1, \quad [h_1,e_{12}]=2e_{12}, \quad [h_1,e_{21}]=-2e_{21}.
\end{equation}
From the Chevalley-Serre relations we see that every triple $\{ e_{ii+1},e_{i+1i},h_i\}$ generates a copy of $\mathfrak{sl}_2$.

\subsection{Embedding of $\mathfrak{sl}_2$ into $\mathcal{R}_3(\mathfrak{h})$}
Consider the algebra $\mathcal{R}_3(\mathfrak{h})$ and consider the adjoint action of $Q_{12}$ on $\mathcal{R}_3(\mathfrak{h})$. We want to find its eigenspaces. The eigenspace with eigenvalue $0$ is a five-dimensional space generated by the central elements $\{Q_1, Q_2, Q_3, Q_{123}\}$ and the operator $Q_{12}$. The eigenvectors with nonzero eigenvalue are
\begin{align}
\begin{split}\label{sl2gen}
 e_{12}&:= \lambda_1([Q_{12},[Q_{12},Q_{23}]]+(a_1+a_2)[Q_{12},Q_{23}] ),\\
 e_{21}&:= \lambda_1([Q_{12},[Q_{12},Q_{23}]]-(a_1+a_2)[Q_{12},Q_{23}] ). 
 \end{split}
\end{align}
We introduced the number
\[
 \lambda_1=\frac{1}{\sqrt{4a_{1}a_{2}a_{3}(a_{1}+a_2)^2(a_{1}+a_2+a_3)}}.
\]
One checks easily using relation \eqref{Dolan-Grady} that:
\begin{align}
\begin{split}\label{eigenvectors}
 [Q_{12}, e_{12}]&=(a_1+a_2) e_{12}, \\
 [Q_{12}, e_{21}]&=-(a_1+a_2)e_{21}.
 \end{split}
\end{align}
We also define the following operator
\begin{equation}\label{sl2h}
h_1:=\frac{2Q_{12}}{a_1+a_2}-\frac{Q_{123}}{a_1+a_2+a_3}-\frac{Q_1}{a_1}-\frac{Q_2}{a_2}+\frac{Q_3}{a_3}.
\end{equation}
We have the following proposition:
\begin{proposition}\label{sl2}
The operators $e_{12}$, $e_{21}$ and $h_1$ satisfy the commutation relations of $\mathfrak{sl}_2$.
\begin{equation*}
 [e_{12},e_{21}]=h_1, \quad [h_1,e_{12}]=2e_{12}, \quad [h_1,e_{21}]=-2e_{21}
\end{equation*}
\end{proposition}
\begin{proof}
One checks through straightforward calculation.
\end{proof}
By Proposition \ref{sl2} we have a map of $\mathfrak{sl}_2 $ into $\mathcal{R}_3(\mathfrak{h})$. Because $\mathfrak{sl}_2$ is simple, the kernel of this map is either trivial of equal to the whole algeba $\mathfrak{sl}_2$. Clearly, $e_{12}$, $e_{21}$ and $h_1$ are different from $0$ so the kernel must be trivial. This map must therefore be injective and we have indeed an embedding of $\mathfrak{sl}_2$ into $\mathcal{R}_3(\mathfrak{h})$.

\subsection{Embedding of $\mathfrak{sl}_3$ into $\mathcal{R}_4(\mathfrak{h})$}
Consider the algebra $\mathcal{R}_4(\mathfrak{h})$. We want to find the common eigenvectors of $Q_{12}$ and $Q_{123}$.
By Lemma \ref{commutator} the operators $e_{12}$, $e_{21}$ and $h_1$ commute with $Q_{123}$ and are therefore eigenvectors of both $Q_{12}$ and $Q_{123}$. 
We can find another set of eigenvectors using Lemma \ref{Racah1}. Consider the operators $e_{12}$, $e_{21}$ and $h_1$ expressed in the operators $Q_{12}$, $Q_{23}$ and $a_{12}$ and replace the indices as follows by Lemma \ref{Racah1}: $1 \rightarrow \{1,2\}$, $2\rightarrow 3$ and $3\rightarrow 4$. We find the following operators:
\begin{align*}
	e_{23}&:=\lambda_2([Q_{123},[Q_{123},Q_{34}]]+(a_1+a_2+a_3)[Q_{123},Q_{34}] ), \\
	e_{32}&:=\lambda_2([Q_{123},[Q_{123},Q_{34}]]-(a_1+a_2+a_3)[Q_{123},Q_{34}]), \\
	h_2&:=\frac{2Q_{123}}{a_1+a_2+a_3}-\frac{Q_{1234}}{a_1+a_2+a_3+a_4}-\frac{Q_{12}}{a_1+a_2}-\frac{Q_3}{a_3}+\frac{Q_4}{a_4}.
\end{align*}
The number $\lambda_2$ is given by
\[
 \lambda_2=\frac{1}{\sqrt{4(a_1+a_{2})a_{3}a_{4}(a_1+a_{2}+a_3)^2(a_1+a_{2}+a_3+a_4)}}.
\]
By Lemma \ref{commutator} the operators $e_{23}$, $e_{32}$ and $h_2$ commute with $Q_{12}$. They are also eigenvectors of $Q_{123}$ by formula \ref{eigenvectors}:
\begin{align}\label{eigenvectors123}
\begin{split}
 [Q_{123}, e_{23}]&=(a_1+a_2+a_3) e_{23}, \\
 [Q_{123}, e_{32}]&=-(a_1+a_2+a_3) e_{32}.
\end{split}
\end{align}
The operators $e_{23}$, $e_{32}$ and $h_2$ satisfy the $\mathfrak{sl}_2$ relations:
\begin{equation*}
 [e_{23},e_{32}]=h_2, \quad [h_2,e_{23}]=2e_{23}, \quad [h_2,e_{32}]=-2e_{32}.
\end{equation*}

We have the following claim:
\begin{proposition}\label{sl3}
The operators $\{ e_{12}, e_{21}, e_{23},e_{32}, h_1,h_2\} $ satisfy the commutation relations of $\mathfrak{sl}_3$.
\end{proposition}
\begin{proof}
We already know that both $\{e_{12}, e_{21}, h_1\}$ and $\{e_{23}, e_{32}, h_2\}$ satisfy the $\mathfrak{sl}_2$ relations. It is also easy to show by Lemma \ref{commutator} that
\[
[h_1,h_2]=0. 
\]
By straightforward calculation
\begin{align*}
[e_{12},e_{32}]=0, \\
[e_{21},e_{23}]=0.
\end{align*}
By now we verified relations \eqref{S1} and \eqref{S2}.
By explicit calculation using the definition of $h_1$ and $h_2$ and formula (\ref{eigenvectors}) and (\ref{eigenvectors123}):
%\begin{align*}
%[h_1,e_{23}]	&=\left[-\frac{Q_{123}}{a_1+a_2+a_3},e_{23}\right] 	&   [h_1,e_{32}]	&=\left[-\frac{Q_{123}}{a_1+a_2+a_3},e_{32}\right] \\
%			&=-e_{23} 								&			&=e_{32} \\
%[h_2,e_{12}]	&=\left[-\frac{Q_{12}}{a_1+a_2},e_{12}\right]		& [h_2,e_{21}]	&=\left[-\frac{Q_{12}}{a_1+a_2},e_{21}\right] \\
%			&=-e_{12} 								& 			&=e_{21} .
%\end{align*}
\begin{align*}
[h_1,e_{23}]	&=\left[-\frac{Q_{123}}{a_1+a_2+a_3},e_{23}\right]=-e_{23}, \\
[h_1,e_{32}]	&=\left[-\frac{Q_{123}}{a_1+a_2+a_3},e_{32}\right] =e_{32}, \\
[h_2,e_{12}]	&=\left[-\frac{Q_{12}}{a_1+a_2},e_{12}\right]=-e_{12}, \\
[h_2,e_{21}]	&=\left[-\frac{Q_{12}}{a_1+a_2},e_{21}\right]=e_{21} . 			
\end{align*}

Relation \eqref{S3} is also satisfied. We only need check \eqref{S+} and \eqref{S-}:
\begin{align*}
[ e_{12},[e_{12},e_{23}]]&=0, \\
[ e_{23},[e_{23},e_{12}]]&=0, \\
[ e_{21},[e_{21},e_{32}]]&=0, \\
[ e_{32},[e_{32},e_{21}]]&=0 .
\end{align*}
This is done by straightforward computation. By Serre's Theorem these are a complete set of defining relations and this concludes the proof.
\end{proof}
By Proposition \ref{sl3} we have a map of $\mathfrak{sl}_3 $ into $\mathcal{R}_4(\mathfrak{h})$. Because $\mathfrak{sl}_3$ is simple, the kernel of this map is either trivial of equal to the whole algeba $\mathfrak{sl}_3$. The generators we used in Proposition \ref{sl3} are different from $0$. This map must therefore be injective and we have indeed an embedding of $\mathfrak{sl}_3$ into $\mathcal{R}_4(\mathfrak{h})$.

\subsection{Embedding of $\mathfrak{sl}_{n-1}$ into $\mathcal{R}_n(\mathfrak{h})$}
As in the previous two sections we construct common eigenvectors for the adjoint action of the Abelian subalgebra $\mathcal{Y}=\{ Q_{[k]} \,| \,2\leq k \leq n-1\}$. Observe that by formula (\ref{aK}) we have  $a_B=\sum_{i \in B} a_i$. To construct the eigenvectors we use Lemma \ref{Racah1} in the following way: Take $e_{12}$, $e_{21}$ and $h_1$ given by formulas \eqref{sl2gen} and \eqref{sl2h} and replace $1$ by $[k]$, $2$ by  $k+1$ and $3$ by  $k+2$.
We obtain the following elements:
\begin{align*}
e_{k k+1}&:=\lambda_k([Q_{[k+1]},[Q_{[k+1]},Q_{k+1 k+2}]]+a_{[k+1]}[Q_{[k+1]},Q_{k+1 k+2}]), \\
e_{k+1 k}&:=\lambda_k([Q_{[k+1]},[Q_{[k+1]},Q_{k+1 k+2}]]-a_{[k+1]}[Q_{[k+1]},Q_{k+1 k+2}] ),\\
h_k&:=\frac{2Q_{[k+1]}}{a_{[k+1]}}-\frac{Q_{[k+2]}}{a_{[k+2]}}-\frac{Q_{[k]}}{a_{[k]}}-\frac{Q_{k+1}}{a_{k+1}}+\frac{Q_{k+2}}{a_{k+2}}.
\end{align*}
We introduced the element
 \[
 \lambda_k=\frac{1}{\sqrt{4a_{[k]}a_{k+1}a_{k+2}a_{[k+1]}^2a_{[k+2]}}}.
 \]
 We check that these operators are indeed eigenvectors of the Abelian algebra $\mathcal{Y}$. By Lemma \ref{commutator} the operator $Q_{[l]}$ commutes with $e_{kk+1}$ and $e_{k+1k}$ if $l \neq k+1$ and with $h_k$ for all $k$. If $l=k+1$ we have
 \begin{align*}
  [Q_{[k+1]},&e_{kk+1}]\\
  =\lambda_k(&[Q_{[k+1]},[Q_{[k+1]},[Q_{[k+1]},Q_{k+1 k+2}]]]+a_{[k+1]}[Q_{[k+1]},[Q_{[k+1]},Q_{k+1 k+2}]]) \\
  =\lambda_k(&a_{[k]}^2[Q_{[k+1]},Q_{k+1 k+2}]+a_{[k+1]}[Q_{[k+1]},[Q_{[k+1]},Q_{k+1 k+2}]]) \\
  =a_{[k]}&e_{kk+1}.
   \end{align*}
 We used formula (\ref{Dolan-Grady}). Similarly one can show that  $[Q_{[k+1]},e_{k+1k}]=-a_{[k]}e_{k+1k}$. We are now ready to prove the following theorem.

\begin{theorem}\label{sln}
The set of operators $\{ e_{i i+1} , e_{ii+1},h_i| i\in [n-1]\} \subset \mathcal{R}_{n}(\mathfrak{h})$  generate an algebra isomorphic to $\mathfrak{sl}_{n-1}$ for the Lie bracket.
\end{theorem}
\begin{proof}
We prove this statement by induction. The cases $n=2$ and $n=3$ have already be obtained in Propositions \ref{sl2} and \ref{sl3}. Now assume that
$\{ e_{i i+1} , e_{ii+1},h_i| i\in [n-2]\} $ generates $\mathfrak{sl}_{n-2}$. To go to $\mathfrak{sl}_{n-1}$ we add three new operators:  $ e_{n-2n-1}$, $e_{n-1n-2}$ and $h_{n-2}$. We introduce a morphism using Lemma \ref{Racah1}:
\[
	\sigma:\mathcal{R}_{n-1}(\mathfrak{h}) \rightarrow \mathcal{R}_{n}(\mathfrak{h}) 
\]
by mapping the indices $1\rightarrow \{1,2\}$ and $i \rightarrow i+1$ for every $i>1$. One should notice that $\sigma(e_{k-1k})=e_{kk+1}$, $\sigma(e_{kk-1})=e_{k+1k}$ and $\sigma(h_{k-1})=h_{k}$. It maps $\mathfrak{sl}_{n-2}$ as a subalgebra of $\mathcal{R}_{n-1}(\mathfrak{h})$ into $\mathcal{R}_{n}(\mathfrak{h}) $. We will use this map a few times.

We check the relations of $\mathfrak{sl}_n$ using Serre's Theorem.
By Lemma \ref{Racah1} and Proposition \ref{sl2} we know that $e_{kk+1}$, $e_{k+1k}$ and $h_k$ generate a $\mathfrak{sl}_2$ algebra for every $k$:
\begin{align*}
[h_k,e_{kk+1}]&=2e_{kk+1}, \\
[h_k,e_{k+1k}]&=-2e_{k+1k}, \\
[e_{kk+1},e_{k+1k}]&=h_k.
\end{align*}
 We also have by Lemma \ref{commutator}:
\[
 [h_k,h_l]=0.
\] 
Relation \eqref{S1} is satisfied so the set $\{ h_k\, |\, k=1\dots n-1\}$ plays the role  Cartan algebra of $\mathfrak{sl}_{n-1}$.
The following relation we need to check is $[e_{kk+1},e_{l+1l}]=0$ if $l \neq k$. If $l \notin \{k-1,k,k+1\}$ this is true by Lemma \ref{commutator}. Otherwise set $l=k+1$. Then we can use the map $\sigma$:
\[
 [e_{kk+1},e_{k+2k+1}]=\sigma( [e_{k-1k},e_{k+1k}])=\sigma(0)=0.
\]
Here we used the relations of the algebra $\mathfrak{sl}_{n-2}$. The proof for $l=k-1$ is analogous. Relation \eqref{S2} is satisfied.
The next commutators we need to calculate are $[h_k,e_{jj+1}]$ and  $[h_k,e_{j+1j}]$:
\begin{align*}
[h_k,e_{jj+1}]&=\left[\frac{2Q_{[k+1]}}{a_{[k+1]}}-\frac{Q_{[k+2]}}{a_{[k+2]}}-\frac{Q_{[k]}}{a_{[k]}},e_{jj+1}\right] \\
	&=2\delta_{k,j}e_{jj+1}-\delta_{k+2,j+1}e_{jj+1}-\delta_{k,j+1}e_{jj+1}, \\
[h_k,e_{j+1j}]&=\left[\frac{2Q_{[k+1]}}{a_{[k+1]}}-\frac{Q_{[k+2]}}{a_{[k+2]}}-\frac{Q_{[k]}}{a_{[k]}},e_{j+1j}\right] \\
	&=-2\delta_{k,j}e_{j+1j}+\delta_{k+2,j+1}e_{j+1j}+\delta_{k,j+1}e_{j+1j}.
\end{align*}
If $k \notin \{j-1,j,j+1\}$ then $[h_k,e_{jj+1}]=0$ and $[h_k,e_{j+1j}]=0$. Otherwise, we have
\begin{align*}
 [h_{j-1},e_{jj+1}]&=-e_{jj+1}, \\
  [h_{j+1},e_{jj+1}]&=-e_{jj+1},\\
   [h_{j-1},e_{j+1j}]&=e_{j+1j}, \\
  [h_{j+1},e_{j+1j}]&=e_{j+1j}.
\end{align*}
We have verified relation \eqref{S3}.
Finally we need to check the relations \eqref{S+} and \eqref{S-}. Specifically, we need to check the following:
\begin{align}
[e_{jj+1},e_{kk+1}]=0, & \quad \text{ if }  k \notin \{  j-1, j+1\} \label{Serre1}, \\
[e_{j+1j},e_{k+1k}]=0, & \quad \text{ if }  k \notin \{  j-1, j+1\} \label{Serre2}, \\
[e_{jj+1},[e_{jj+1},e_{kk+1}]]=0, & \quad \text{ if }  k \in \{  j-1, j+1\} \label{Serre3}, \\
[e_{j+1j},[e_{j+1j},e_{k+1k}]]=0, & \quad \text{ if }  k \in \{  j-1, j+1\}  \label{Serre4}.
\end{align}
The relations (\ref{Serre1}) and (\ref{Serre2}) follow by Lemma \ref{commutator}. The relations (\ref{Serre3} )and (\ref{Serre4}) can be proven as follows. 
\begin{align*}
[e_{jj+1},[e_{jj+1},e_{kk+1}]]&=\sigma([e_{j-1j},[e_{j-1j},e_{k-1k}]])=\sigma(0)=0,\\
[e_{j+1j},[e_{j+1j},e_{k+1k}]]&=\sigma([e_{jj-1},[e_{jj-1},e_{kk-1}]])=\sigma(0)=0.
\end{align*}
Here we used the $\mathfrak{sl}_{n-2}$ relations. 
We have shown that the algebra generated by  $\{ e_{i i+1} , e_{ii+1},h_i| i\in [n-1]\}$ is homomorphic to $\mathfrak{sl}_{n-1}$. To show that it is in fact isomorphic we need to prove that the map from  $\mathfrak{sl}_{n-1}$ into $\mathcal{R}_n(\mathfrak{h})$ is injective. This is seen by the same argument as before based on the simplicity of $\mathfrak{sl}_{n-1}$ and thus we have an embedding
\end{proof}

We state the following corollary.
\begin{corollary}\label{isomorph} As Lie algebras over the field $\mathbb{K}:=\mathbb{R}(a_1,\dots, a_n)$ we have the following isomorphism:
\[
\mathcal{R}_n(\mathfrak{h},[])\cong \mathfrak{sl}_{n-1}\oplus \mathbb{K}^{n+1}.
\]
\end{corollary}
\begin{proof}
The set (\ref{engendre}) generates $\mathcal{R}_n(\mathfrak{h},[])$ as a vector space over $\mathbb{K}:=\mathbb{R}(a_1,\dots, a_n)$ but we do not know yet if this is a basis. We have $n+1$ central elements, $\frac{n(n-1)}{2}$ elements of the form $Q_{ij}$ and $\frac{(n-1)(n-2)}{2}$ elements of the form $[Q_{1j},Q_{jk}]$. We have therefore
\begin{equation}\label{upperboundRacah}
 \dim(\mathcal{R}_n(\mathfrak{h},[]))\leq n+1+ \frac{n(n-1)}{2}+\frac{(n-1)(n-2)}{2}=n^2-n+1.
\end{equation}
By Theorem \ref{sln} we know there exists a subalgebra of $\mathcal{R}_n(\mathfrak{h},[])$ isomorphic with $\mathfrak{sl}_{n-1}$. Denote this Lie algebra by $\mathfrak{s}$. The dimension of $\mathfrak{s}$ equals $(n-1)^2-1$. We add to this algebra $\mathfrak{s}$ the vector space generated by the central elements 
\[Z:=\langle\{Q_i\,|\,i \in [n]\} \cup \{Q_{[n]}\}\rangle. \]
First notice that $\mathfrak{s} \cap Z= \{0\}$. The algebra $\mathfrak{s}$ is simple so it does not contain central elements. Second we see that $\{Q_i\,|\,i \in [n]\} \cup \{Q_{[n]}\}$ is a basis for $Z$. If this was not the case, we could find a linear combination of the central elements equal to $0$. We show that this is not possible. The $Q_i$ are linearly independent by construction. Assume we have the linear combination:
\[ 
 Q_{[n]}=\sum_{i=1}^n\lambda_i Q_i.
\]
Act with $\mu^*\otimes\bigotimes_{i=1}^{n-1}1$ and take the commutator with $Q_{23}$. We find
\begin{align*}
0= [Q_{23},Q_{[n+1]}]&=\left[ Q_{23},\lambda_1Q_{12}+\sum_{i=2}^n\lambda_i Q_{i+1}\right] \\
 					&=\lambda_1[Q_{23},Q_{12}].
\end{align*}
The element $[Q_{23},Q_{12}]$ does not equal $0$ so $\lambda_1=0$. Similarly one can show that all $\lambda_i=0$ concluding that $\{Q_i \, |\, i \in [n]\} \cup \{Q_{[n]}\}$ is indeed a basis for $Z$. The dimension of $Z$ equals $n+1$. We have 
\[\mathfrak{s} \oplus Z \subset \mathcal{R}_n(\mathfrak{h},[]). \]
From this follows that 
\[
 n^2-n+1=\dim(\mathfrak{s} \oplus Z) \leq \dim(\mathcal{R}_n(\mathfrak{h},[])).
\]
We already have an upper bound by inequality (\ref{upperboundRacah}). Therefore it follows that $\dim(\mathcal{R}_n(\mathfrak{h},[]))=n^2-n+1$ and  hence
\[
 \mathcal{R}_n(\mathfrak{h},[])=\mathfrak{s}\oplus Z \cong \mathfrak{sl}_{n-1} \oplus \mathbb{K}^{n+1}.
\]
This concludes the proof.
\end{proof}
Theorem \ref{sln} and Corollary \ref{isomorph} both follow from the set of equations calculated in section \ref{sectiondef}. As a consequence of the isomorphism found in Corollary \ref{isomorph} we give an alternative definition of  $\mathcal{R}_n(\mathfrak{h},[])$. It is the algebra defined over $\mathbb{R}(a_1,\dots,a_n)$ with generators $Q_i$ and $Q_{ij}$ satisfying the set of equalities \eqref{Commsum}, \eqref{relation1} and \eqref{relation2},  together with Lemma \ref{commutator} and Lemma \ref{LinearDependence}.

\subsection{Labelling Abelian algebras}
In the proof of Theorem \ref{sln} the elements of the Cartan algebra of $\mathfrak{sl}_{n-1}$ are linear combinations of the central elements and $\{ Q_{[k]} \,|\, 2 \leq k \leq n-1\}$. It is possible to construct different Cartan algebras if one starts from a different set of generators.
\begin{definition}
A set of non-empty subsets of $[n]$ $\mathcal{A}$ is a maximal non-intersecting/ nested set if it satisfies the following properties:
\begin{itemize}
 \item  Every pair of sets in $\mathcal{A}$ is either disjoint or one is included in the other. 
 \item  Maximality: It is not possible to add another subset of $[n]$ to $\mathcal{A}$ without contradicting the first property.
\end{itemize}
\end{definition}
By Lemma \ref{commutator} the corresponding generators $Q_A$ with $A \in \mathcal{A}$ will be a set of commuting operators. To each maximally non-intersecting/nested set $\mathcal{A}$ we can associate a graph $\mathbb{T}_\mathcal{A}$. Let every vertex of the graph represent a set in $\mathcal{A}$. There is an edge between to vertices $A$ and $B$ if either set is included in the other but there is no set $C \in \mathcal{A}$ such that $A \subset C \subset B$ or $B \subset C \subset A$.
\begin{lemma}
The graph $\mathbb{T}_\mathcal{A}$ related to the maximal non-intersecting/nested set $\mathcal{A}$ is a perfect binary tree.
\end{lemma}
\begin{proof}
The set $[n]$ is always included in $\mathcal{A}$ by maximality. The set $[n]$ includes every set in $\mathcal{A}$ and will be the root of the tree. The sets with one element are also included in $\mathcal{A}$ by maximality. They are either completely included or disjoint with every set in $\mathcal{A}$. These sets do not have any proper subsets so they cannot have any children and are therefore leaves of the tree. Every other vertex has two children. Assume this is not the case. Let $A$ have exactly one child $B$. Then we can add $A\backslash B$ to $\mathcal{A}$ contradicting the maximality of $\mathcal{A}$. Assume $A$ has more than two children. Let $B_1$, $B_2$ and $B_3$ three sets that are children of $A$. We can add the set $B_1 \cup B_2$ to $\mathcal{A}$ contradicting the maximality of $\mathcal{A}$. This concludes the proof.
\end{proof}
We have $n$ indices so we have $n$ leaves. A perfect binary tree with $n$ leaves has $n-2$ interior vertices. The generators related to the interior vertices constitute an Abelian algebra with $n-2$ generators.
\begin{definition}
Let $\mathcal{A}$ be a maximally non-intersecting/nested set of $[n]$. Then we define the labelling Abelian algebra associated to $\mathcal{A}$ to be:
\[
\mathcal{Y}_{\mathcal{A}}=\{ Q_A  \,| \, A \in \mathcal{A} \text{ and } 1<|A|<n\}.
\]
\end{definition}
We exclude the elements $Q_i$ and $Q_{[n]}$ as they are central. For each maximally non-intersecting/nested set $\mathcal{A}$ we have constructed a tree $\mathbb{T}_\mathcal{A}$ and an algebra $\mathcal{Y}_\mathcal{A}$. Often we will represent the algebra  $\mathcal{Y}_\mathcal{A}$ by the tree $\mathbb{T}_\mathcal{A}$. These trees are similar to the coupling trees introduced in \cite{Vanderjeugt-2003}.
\begin{example}
Consider $\mathcal{R}_3(\mathfrak{h})$. We have three indices $1$, $2$ and $3$. This gives us three possible trees and therefore three different labelling Abelian algebras:
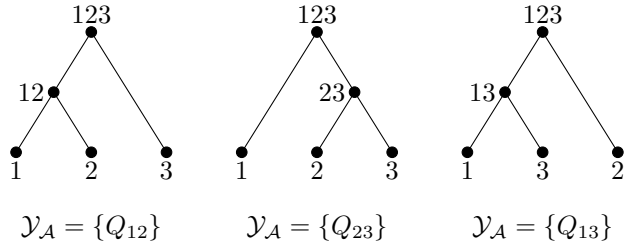
\begin{center}
\begin{tikzpicture}
\filldraw 	(0,0) circle (2pt) node[anchor=north] {1}
		(0.5,0.8) circle (2pt) node[anchor=east] {12}
		(1,1.6) circle (2pt)node[anchor=south]{123}
		(1,0) circle (2pt) node[anchor=north]{2}
		(2,0) circle (2pt) node[anchor=north]{3};
\draw (0,0)--(1,1.6);
\draw (0.5,0.8)--(1,0);
\draw (1,1.6)--(2,0);
\draw (1,-1)node {$\mathcal{Y}_{\mathcal{A}}=\{Q_{12}\}$};

\filldraw 	(3,0) circle (2pt) node[anchor=north] {1}
		(4.5,0.8) circle (2pt) node[anchor=east] {23}
		(4,1.6) circle (2pt)node[anchor=south]{123}
		(4,0) circle (2pt) node[anchor=north]{2}
		(5,0) circle (2pt) node[anchor=north]{3};
\draw (3,0)--(4,1.6);
\draw (4.5,0.8)--(4,0);
\draw (4,1.6)--(5,0);
\draw (4,-1)node {$\mathcal{Y}_{\mathcal{A}}=\{Q_{23}\}$};

\filldraw 	(6,0) circle (2pt) node[anchor=north] {1}
		(6.5,0.8) circle (2pt) node[anchor=east] {13}
		(7,1.6) circle (2pt)node[anchor=south]{123}
		(7,0) circle (2pt) node[anchor=north]{3}
		(8,0) circle (2pt) node[anchor=north]{2};
\draw (6,0)--(7,1.6);
\draw (6.5,0.8)--(7,0);
\draw (7,1.6)--(8,0);
\draw (7,-1)node {$\mathcal{Y}_{\mathcal{A}}=\{Q_{13}\}$};
		
\end{tikzpicture}
\end{center}
\end{example}

\begin{example}
Consider $\mathcal{R}_4(\mathfrak{h})$. We have four indices $1$, $2$, $3$ and $4$ so we have trees with four leaves. We consider two examples:

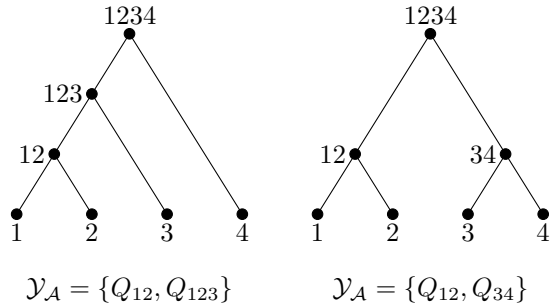
\begin{center}
\begin{tikzpicture}
\filldraw 	(0,0) circle (2pt) node[anchor=north] {1}
		(0.5,0.8) circle (2pt) node[anchor=east] {12}
		(1,1.6) circle (2pt)node[anchor=east]{123}
		(1,0) circle (2pt) node[anchor=north]{2}
		(2,0) circle (2pt) node[anchor=north]{3}
		(1.5,2.4) circle (2pt) node[anchor=south] {1234}
		(3,0) circle (2pt) node[anchor=north]{4};
\draw (0,0)--(1.5,2.4);
\draw (0.5,0.8)--(1,0);
\draw (1,1.6)--(2,0);
\draw (1.5,2.4)--(3,0);
\draw (1.5,-1)node {$\mathcal{Y}_{\mathcal{A}}=\{Q_{12},Q_{123}\}$};

\filldraw 	(4,0) circle (2pt) node[anchor=north] {1}
		(4.5,0.8) circle (2pt) node[anchor=east] {12}
		(6.5,0.8) circle (2pt)node[anchor=east]{34}
		(5,0) circle (2pt) node[anchor=north]{2}
		(6,0) circle (2pt) node[anchor=north]{3}
		(5.5,2.4) circle (2pt) node[anchor=south] {1234}
		(7,0) circle (2pt) node[anchor=north]{4};
\draw (4,0)--(5.5,2.4);
\draw (4.5,0.8)--(5,0);
\draw (6.5,0.8)--(6,0);
\draw (5.5,2.4)--(7,0);
\draw (5.5,-1)node {$\mathcal{Y}_{\mathcal{A}}=\{Q_{12},Q_{34}\}$};

\end{tikzpicture}
\end{center}
\end{example}
We are now in a position to explain the relation between these labelling Abelian algebras and the Cartan algebra of $\mathfrak{sl_{n-1}}$ inside $\mathcal{R}_n(\mathfrak{h})$. Assume we have fixed a labelling Abelian algebra $\mathcal{Y}_{\mathcal{A}}$. For every set $A \in \mathcal{A}$ that is not the root or a leaf, we will construct three generators $e_A$, $f_A$ and $h_A$. To do this consider the tree $\mathbb{T}_\mathcal{A}$. Focus on the subtree consisting of the vertex related to the set $A$: The children of $A$ which we will call $K$ and $L$, the parent $B$ of $A$ and the child of $B$ differing from $A$ denoted by $M$. We have two possibilities:
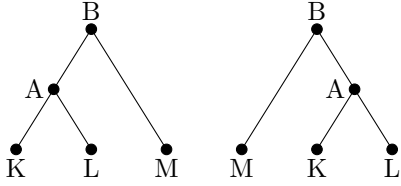
\begin{center}
\begin{tikzpicture}
\filldraw 	(0,0) circle (2pt) node[anchor=north] {K}
		(0.5,0.8) circle (2pt) node[anchor=east] {A}
		(1,1.6) circle (2pt)node[anchor=south]{B}
		(1,0) circle (2pt) node[anchor=north]{L}
		(2,0) circle (2pt) node[anchor=north]{M};
\draw (0,0)--(1,1.6);
\draw (0.5,0.8)--(1,0);
\draw (1,1.6)--(2,0);

\filldraw 	(3,0) circle (2pt) node[anchor=north] {M}
		(4.5,0.8) circle (2pt) node[anchor=east] {A}
		(4,1.6) circle (2pt)node[anchor=south]{B}
		(4,0) circle (2pt) node[anchor=north]{K}
		(5,0) circle (2pt) node[anchor=north]{L};
\draw (3,0)--(4,1.6);
\draw (4.5,0.8)--(4,0);
\draw (4,1.6)--(5,0);
\end{tikzpicture}
\end{center}
We define $e_A$ and $f_A$ as follows:
\begin{align}\label{sl2A}
\begin{split}
e_A:=\lambda_A([Q_{A},[Q_{A},Q_{LM}]]+a_A[Q_{A},Q_{LM}] ), \\
f_A:=\lambda_A([Q_{A},[Q_{A},Q_{LM}]]-a_A[Q_{A},Q_{LM}] ).
\end{split}
\end{align}
with 
\[
 \lambda_A=\frac{1}{\sqrt{4a_Ka_La_Ma_{A}^2a_{B}}}.
\]
Observe that it is possible to replace $Q_{LM}$ by $Q_{KM}$ in the definitions of $e_A$ and $f_A$ leading to different but equally correct expressions for $e_A$ and $f_A$. To avoid this ambiguity we always choose $L$ to be the right child of $A$ and $M$ the child of $B$ different from $A$. This, however, fixes an ordering of the leaves. Otherwise it is not possible to speak of a left and right child.
Together with the Cartan elements $h_A=[e_A,f_A]$ the set $\{ e_A, f_A ,h_A\}$ generates the Lie algebra $\mathfrak{sl}_{n-1}$. This can be proven by repeating the arguments in  Theorem \ref{sln}. The maximally non-intersecting/nested set used in Theorem \ref{sln} is $ \{ [k]\,| \,1<k<n\}$.
To find the relation between the Cartan algebra $\{h_A\,|\, A \in \mathcal{A} \text{ and } 1<|A|<n \}$ and the labelling Abelian algebra $\mathcal{Y}_{\mathcal{A}}$ we consider explicitly $h_A$:
\[
h_{A}=\frac{2Q_A}{a_A}-\frac{Q_B}{a_B}-\frac{Q_{K}}{a_{K}}-\frac{Q_{L}}{a_{L}}+\frac{Q_{M}}{a_{M}}.
\]
Every Cartan element can be written as a linear combination of elements of the labelling Abelian algebra $\mathcal{Y}_{\mathcal{A}}$ and central elements $Q_i$ and $Q_{[n]}$.

In the next section we will study the representation theory behind $\mathcal{R}_n(\mathfrak{h})$ and its relation to the Lie algebra $\mathfrak{sl}_{n-1}$. The relation between Cartan algebras and labelling Abelian algebras will play an important role there.

\section{Connection between recoupling coefficients for $\mathfrak{h}$ and $\mathfrak{sl}_{n-1}$-representations} \label{representationtheory}
Assume we have a finite dimensional irreducible representation $V$ of $\mathcal{R}_n(\mathfrak{h})$. Because of corollary \ref{isomorph} this is also an irreducible representation for $\mathfrak{sl}_{n-1}$. In this section we study bases diagonalized by different labelling Abelian algebas or equivalently Cartan algebras of $\mathfrak{sl}_{n-1}$ and their connection coefficients. We first study the rank one case.

\subsection{$\mathfrak{sl}_2$ and the Krawtchouk polynomials}
Consider the following two Abelian algebras of $\mathcal{R}_3(\mathfrak{h})$:
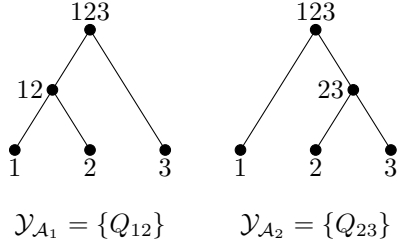
\begin{center}
\begin{tikzpicture}
\filldraw 	(0,0) circle (2pt) node[anchor=north] {$1$}
		(0.5,0.8) circle (2pt) node[anchor=east] {$12$}
		(1,1.6) circle (2pt)node[anchor=south]{$123$}
		(1,0) circle (2pt) node[anchor=north]{$2$}
		(2,0) circle (2pt) node[anchor=north]{$3$};
\draw (0,0)--(1,1.6);
\draw (0.5,0.8)--(1,0);
\draw (1,1.6)--(2,0);
\draw (1,-1)node {$\mathcal{Y}_{\mathcal{A}_1}=\{Q_{12}\}$};

\filldraw 	(3,0) circle (2pt) node[anchor=north] {$1$}
		(4.5,0.8) circle (2pt) node[anchor=east] {$23$}
		(4,1.6) circle (2pt)node[anchor=south]{$123$}
		(4,0) circle (2pt) node[anchor=north]{$2$}
		(5,0) circle (2pt) node[anchor=north]{$3$};
\draw (3,0)--(4,1.6);
\draw (4.5,0.8)--(4,0);
\draw (4,1.6)--(5,0);
\draw (4,-1)node {$\mathcal{Y}_{\mathcal{A}_2}=\{Q_{23}\}$};
		
\end{tikzpicture}
\end{center}
Let $\{ \psi_k \}$ be an eigenbasis of $Q_{12}$ and $\{ \phi_s \}$ be an eigenbasis of $Q_{23}$. The indices $k$ and $s$ run form $0$ to $N$ with $\dim(V)=N+1$. We are interested in the overlap coefficients $B_{ks}$ between these bases:
\begin{equation}\label{overlapsl2}
 \phi_s=\sum_{k=0}^N B_{sk}\psi_k.
\end{equation}
To study these overlap coefficients we construct the Lie algebra $\mathfrak{sl}_2$ inside $\mathcal{R}_3(\mathfrak{h})$. Corresponding to $\mathcal{Y}_{\mathcal{A}_1}$, we have the following generators. See also formula (\ref{sl2gen}):
\begin{align*}
 e&:=\lambda([Q_{12},[Q_{12},Q_{23}]]+a_{12}[Q_{12},Q_{23}]), \\
 f&:=\lambda([Q_{12},[Q_{12},Q_{23}]]-a_{12}[Q_{12},Q_{23}]), \\
 h&:=\frac{2Q_{12}}{a_{12}}-\frac{Q_{123}}{a_{123}}-\frac{Q_1}{a_1}-\frac{Q_2}{a_2}+\frac{Q_3}{a_3}
 \end{align*}
 with $\lambda=\sqrt{4a_1a_2 a_3 a_{12}^2 a_{123}}^{-1}$.
 The $\mathfrak{sl}_2$ Lie algebra related to $\mathcal{Y}_{\mathcal{A}_2}$ is given by replacing indices $(1 \rightarrow 2 \rightarrow 3 \rightarrow 1)$ in $e$, $f$ and $h$ given above. We denote these elements by $\tilde{e}$, $\tilde{f}$ and $\tilde{h}$. The elements $h$ and $\tilde{h}$ diagonalize the bases $\{ \psi_k \}$ and $\{ \phi_s \}$ respectively. 
  \[
    h\psi_k=\mu_k\psi_k, \qquad \tilde{h}\phi_s=\nu_s\phi_s.
 \]
The map $\tilde{.}$ is an automorphism of $\mathfrak{sl}_2$ so we can write $\tilde{h}$ as a linear combination of $h$, $e$ and $f$. 
\[
  \tilde{h}=R_h h +R_e e+ R_f f
\]
with
\begin{align*}
 R_h&=-\frac{a_1a_2-a_1a_3+a_2^2+a_2a_3}{a_{12}a_{23}}, \\
 R_e&=R_f=2\frac{\sqrt{a_1a_2a_3a_{123}}}{a_{12}a_{23}} .
\end{align*}
Observe that $R_eR_f+R_h^2=1$.
It is a classical result that the overlap coefficients between two bases related by an inner automorphism of $\mathfrak{sl}_2$ are univariate Krawtchouk polynomials \cite{GVZ5,Koorn3}. We have therefore relegated the calculations to Appendix \ref{Overlap}. This is the result:
\[
  B_{sk} \sim K_k\left(\frac{\nu_s+N}{2}; \frac{1-R_h}{2},N\right).
\]
The Krawtchouk polynomials depend on $2$ values:  $R_h$ and $N$. The number $N$ is equal by definition to $\dim(V)-1$. The dimension of an irreducible representation of $\mathfrak{sl}_2$ can be found by considering its Casimir:
\begin{align}\label{Casimir-sl2}
\begin{split}
 C&:=\frac{h^2}{2}+ef+fe \\
 	&=\frac{1}{2}\left(\frac{Q_{1}}{a_1}+\frac{Q_2}{a_2}+\frac{Q_3}{a_3}-\frac{Q_{123}}{a_{123}}\right)\left(\frac{Q_{1}}{a_1}+\frac{Q_2}{a_2}+\frac{Q_3}{a_3}-\frac{Q_{123}}{a_{123}}-2\right).
\end{split}
\end{align}
If we act with the Casimir on $\psi_k$ we find:
\[
 C\psi_k=\frac{N^2+2N}{2}\psi_k.
\]
From the action of the Casimir $C$ of $\mathfrak{sl}_2$ or equivalently the central elements $Q_1$, $Q_2$, $Q_3$ and $Q_{123}$ we are able to discern the dimension of the representation and hence the number $N$. For the remainder of the article we will write the dependence on $R_h$ and $C$ explicitly: $B_{sk}(R_h,C)$.

\subsection{$\mathfrak{sl}_{n-1}$ and the multivariate Krawtchouk polynomials}
Let $\mathcal{Y}_{\mathcal{A}_1}$ and  $\mathcal{Y}_{\mathcal{A}_2}$ be two labelling Abelian algebras of $\mathcal{R}_n(\mathfrak{h})$. Additionally we demand that $\mathcal{A}_1$ and $\mathcal{A}_2$ differ by only one element:
\[
   \mathcal{A}_1 \backslash \mathcal{A}_2=\{G_1\}, \quad \mathcal{A}_2 \backslash \mathcal{A}_1=\{G_2\}.
\]
Let $\{\psi_{\vec k} \}$ be diagonalized by the labelling Abelian algebra $\mathcal{Y}_{\mathcal{A}_1}$ and $\{ \psi_{\vec s} \}$ be diagonalized by the labelling Abelian algebra $\mathcal{Y}_{\mathcal{A}_2}$. Let $\mathcal{A}_1=\{A_i \,|\, i=1\dots n-2\}$, then we have
\[
 Q_{A_i}\psi_{\vec k}=\mu^{A_i}_{k_i}\psi_{\vec k}.
\]
For a specific index $j$ it must be that $A_j=G_1$. We want to find the overlap coefficients between the bases $\{\psi_{\vec k} \}$ and $\{\phi_{\vec s} \}$. 
\[
 \phi_{\vec s}=\sum_{\vec k} B_{\vec s \vec k} \psi_{\vec k }.
\]
The basis vector $\phi_{ \vec s}$ is a common eigenvector of the operators $Q_{A_i}$, $i \neq j$ with eigenvalues $\mu^{A_i}_{s_i}$. The basis $\{\psi_{\vec k} \}$ also consists of eigenvectors of the operators $Q_{A_i}$, $i \neq j$. The vector  $\phi_{\vec s}$ must therefore be written as a linear combination of common eigenvectors of $Q_{A_i}$, $i \neq j$  with the same eigenvalues $\mu^{A_i}_{s_i}$, $i \neq j$. It follows that $B_{\vec s \vec k}=0$ if $k_i\neq s_i$ for some $i \neq j$. The overlap coefficient can be written as
\[
B_{\vec s \vec k}=B_{s_jk_j}\prod_{i \neq j} \delta_{s_ik_i} 
\]
where $\delta_{s_ik_i} $ is the Kronecker delta. To find $B_{s_jk_j}$ consider the common eigenspace $T=\{ v\in V\,| \, Q_{A_i}v=\mu^{A_i}_{s_i}v \text { for all } i \neq j\}$. Both $Q_{G_1}$ and $Q_{G_2}$ commute with each $Q_{A_i}$ so they preserve the common eigenspace $T$. In fact $Q_{G_1}$ and $Q_{G_2}$ lie in an algebra isomorphic to $\mathcal{R}_3(\mathfrak{h})$ that preserves $T$. Let $K=G_1\backslash G_2$, $L=G_1 \cap G_2$ and $M=G_2 \backslash G_1$. By Lemma \ref{Racah1} the algebra $\mathcal{R}_3^{K,L,M}(\mathfrak{h})$ generated by $\{ Q_K, Q_L, Q_M, Q_{KL},Q_{LM},Q_{KLM}\}$ is isomorphic to $\mathcal{R}_3(\mathfrak{h})$. It preserves $T$ as each generator commutes with $Q_{A_i}$ with $i \neq j$. In fact the sets $K$, $L$, $M$ and $K\cup L\cup M$ are in $\mathcal{A}_1 \cap \mathcal{A}_2$. We conclude that $T$ is a representation of  $\mathcal{R}_3(\mathfrak{h})$ with one basis $\{ \psi_{\vec k}\} \cap T$ diagonalized by $Q_{G_1}$ and the other $\{\phi_{\vec s}\} \cap T$ by $Q_{G_2}$. We are basically in the situation discussed in the previous paragraph and represented by the following two trees:
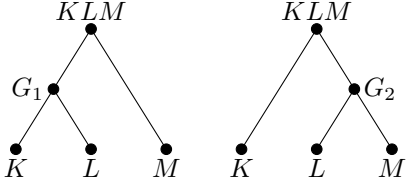
\begin{center}
\begin{tikzpicture}
\filldraw 	(0,0) circle (2pt) node[anchor=north] {$K$}
		(0.5,0.8) circle (2pt) node[anchor=east] {$G_1$}
		(1,1.6) circle (2pt)node[anchor=south]{$KLM$}
		(1,0) circle (2pt) node[anchor=north]{$L$}
		(2,0) circle (2pt) node[anchor=north]{$M$};
\draw (0,0)--(1,1.6);
\draw (0.5,0.8)--(1,0);
\draw (1,1.6)--(2,0);

\filldraw 	(3,0) circle (2pt) node[anchor=north] {$K$}
		(4.5,0.8) circle (2pt) node[anchor=west] {$G_2$}
		(4,1.6) circle (2pt)node[anchor=south]{$KLM$}
		(4,0) circle (2pt) node[anchor=north]{$L$}
		(5,0) circle (2pt) node[anchor=north]{$M$};
\draw (3,0)--(4,1.6);
\draw (4.5,0.8)--(4,0);
\draw (4,1.6)--(5,0);
		
\end{tikzpicture}
\end{center}
This means that the overlap coefficients are given by 
\[
 B_{s_jk_j}= B_{s_jk_j}(R_h^{K,L,M},C_{K,L,M}).
\]
The elements $R_h^{K,L,M}$ and $C_{K,L,M}$ are obtained by using the isomorphism between $\mathcal{R}_3(\mathfrak{h})$ and $\mathcal{R}_3^{K,L,M}(\mathfrak{h})$ obtained by replacing the indices $1$, $2$ and $3$ by $K$, $L$ and $M$:
\begin{align*}
R_h^{K,L,M}&=-\frac{a_Ka_L-a_Ka_M+a_L^2+a_La_M}{a_{KL}a_{LM}}, \\
C^{K,L,M}&=\frac{1}{2}\left(\frac{Q_{K}}{a_K}+\frac{Q_L}{a_L}+\frac{Q_M}{a_M}-\frac{Q_{KLM}}{a_{KLM}}\right)\left(\frac{Q_{K}}{a_K}+\frac{Q_L}{a_L}+\frac{Q_M}{a_M}-\frac{Q_{KLM}}{a_{KLM}}-2\right).
\end{align*}

 We have shown that the overlap coefficients between two bases diagonalized by labelling Abelian algebras differing by one generators are Krawtchouk polynomials. Let us remove the condition that the labelling Abelian algebras need to differ by one generator. Then the strategy to find the overlap coefficients is to find a series of intermediate bases in such a way that each intermediate basis differs by one generator with the next intermediate basis. For example take the labelling Abelian algebras $\langle Q_{12},Q_{34} \rangle$ and $\langle Q_{13},Q_{24} \rangle$ in $\mathcal{R}_4(\mathfrak{h})$. Then we can find a series of intermediate bases:
 \[
 \langle Q_{12},Q_{34} \rangle \rightarrow \langle Q_{12},Q_{123} \rangle \rightarrow \langle Q_{13},Q_{123} \rangle \rightarrow \langle Q_{13},Q_{24} \rangle.
 \]
 Each step gives us Krawtchouk polynomials as overlaps. More specifically let $\{\psi^1_{k_1k_2} \}$, $\{\psi^2_{k_1l_2} \}$, $\{\psi^3_{s_1l_2} \}$ and $\{\phi_{s_1s_2} \}$ be the bases diagonalized by each step in this chain. Then the overlap coefficients become:
 \begin{align*}
   \phi_{s_1s_2}& =\sum_{l_2} B_{s_2l_2}(R_h^{13,2,4},C^{13,2,4}) \psi^3_{s_1l_2} \\
  			&=\sum_{l_2 k_1} B_{s_2l_2}(R_h^{13,2,4},C^{13,2,4})B_{s_1k_1}(R_h^{1,2,3},C^{1,2,3}) \psi^2_{k_1l_2} \\
			&=\sum_{l_2 k_1 k_2} B_{s_2l_2}(R_h^{13,2,4},C^{13,2,4})B_{s_1k_1}(R_h^{1,2,3},C^{1,2,3}) B_{l_2 k_2}(R_h^{4,3,12},C^{4,3,12})\psi^1_{k_1k_2} .
 \end{align*}
 Here we conclude that the overlap coefficients are 
 \begin{equation}\label{overlapsl4}
  B_{\vec s \vec k}=\sum_{l_2} B_{s_2l_2}(R_h^{13,2,4},C^{13,2,4})B_{s_1k_1}(R_h^{1,2,3},C^{1,2,3}) B_{l_2 k_2}(R_h^{4,3,12},C^{4,3,12}).
 \end{equation}
 This gives us a method to calculate connection coefficients between any pair of bases diagonalized by labelling Abelian algebras. We do need to check if there is always a path between two labelling Abelian algebras. To this end we introduce the recoupling graph of $\mathcal{R}_n(\mathfrak{h})$. Let every labelling Abelian algebra be represented by a vertex. Two vertices are connected by an edge if they only differ by one generator.

 \begin{proposition}\label{connectiongraph}
 The recoupling graph of $\mathcal{R}_n(\mathfrak{h})$ is connected and its diameter is bounded by $\frac{(n-1)(n-2)}{2}$.
 \end{proposition}
\begin{proof}
We use an argument applied on the binary trees related to labelling Abelian algebras. On these trees we have two operations:
\begin{itemize}
	\item A twist: interchanging two children (and their descendant trees) of a vertex.  This does not change the labelling Abelian algebra so on the connection graph we stay on the same vertex.
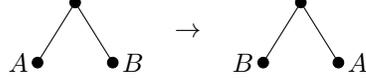
\begin{center}
\begin{tikzpicture}
\filldraw 	(0,0) circle (2pt) node[anchor=east]{$A$}
		(0.5,0.8) circle (2pt) 
		(1,0) circle (2pt)node[anchor=west]{$B$};
\draw (0,0)--(0.5,0.8);
\draw (0.5,0.8)--(1,0);

\draw (2,0.4) node{$\rightarrow$};

\filldraw 	(3,0) circle (2pt) node[anchor=east]{$B$}
		(3.5,0.8) circle (2pt) 
		(4,0) circle (2pt)node[anchor=west]{$A$};
\draw (3,0)--(3.5,0.8);
\draw (3.5,0.8)--(4,0);

\end{tikzpicture}
\end{center}

	\item A swap: Moving a vertex' child to the other edge. This changes the labelling Abelian algebra by one generator. On the connection graph we move along an edge to a new vertex.
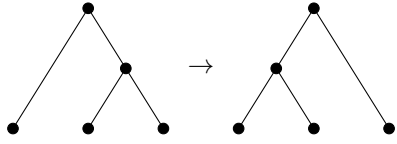
\begin{center}
\begin{tikzpicture}
\filldraw 	(0,0) circle (2pt) 
		(1.5,0.8) circle (2pt) 
		(1,1.6) circle (2pt)
		(1,0) circle (2pt) 
		(2,0) circle (2pt) ;
\draw (0,0)--(1,1.6);
\draw (1.5,0.8)--(1,0);
\draw (1,1.6)--(2,0);

\draw (2.5, 0.8) node {$\rightarrow$};

\filldraw 	(3,0) circle (2pt) 
		(3.5,0.8) circle (2pt) 
		(4,1.6) circle (2pt)
		(4,0) circle (2pt) 
		(5,0) circle (2pt) ;
\draw (3,0)--(4,1.6);
\draw (3.5,0.8)--(4,0);
\draw (4,1.6)--(5,0);
		
\end{tikzpicture}
\end{center}
\end{itemize}
Because twisting does not change the labelling Abelian algebra, you stay on the same vertex of the connection graph. Every swap on the other hand is related to a step on the connection graph. The proof of connectedness in $\frac{(n-2)(n-1)}{2}$ steps uses induction:

Let $n=3$ and take any two trees with three leaves. It is easy to see that it takes a single swap combined with a number of twists to get from one tree to the other.

Assume we have proved it for $n-1$. We take two trees with $n$ leaves. We call these trees the initial tree and final tree. In the initial tree there is at least one vertex whose children are leaves. Assume these leaves are labelled $a$ and $b$. Remove those leaves and give the parent vertex the label $a$. We now have a tree with $n-1$ labelled leaves. We removed the label $b$. Remove this label $b$ also from the final tree together with its parent. The final tree now also has $n-1$ labelled leaves. By induction it takes $\frac{(n-2)(n-3)}{2}$ swaps and a number of twists to change the initial tree to the final tree. Add the leaf with label $b$ and its parent again to the final tree where it was removed. In the initial tree we add two leaves to the leaf with label $a$. We remove this label $a$ and add label $a$ and $b$ to the leaves.
Only leaf $b$ is in the wrong place. There are at most $n-2$ vertices in between leaf $b$ and were it needs to be in the final tree. It requires n-2 swaps and a number of twists to move leaf $b$ into the right position.
The total swaps used to change the initial tree into the final tree is 
\[
\frac{(n-2)(n-3)}{2}+n-2=\frac{(n-1)(n-2)}{2}.
\]
\end{proof}
\begin{remark}
This proof also works for the higher rank Racah algebra for $\mathfrak{su}(1,1)$ as in \cite{DeBie&Genest$vandeVijver&Vinet} effectively generalizing the connection graph and the proof of connectedness of the connection graph in \cite{DeBie&Genest$vandeVijver&Vinet}.
\end{remark}

We want to conclude this section with a special pair of labelling Abelian algebras of $\mathcal{Y}_1=\{ Q_{[k]}\,|\, 2\leq k \leq n-1\}$ and $\mathcal{Y}_{n-1}=\{ Q_{[2\dots k]}\,|\, 3\leq k \leq n\}$. Observe that $\mathcal{Y}_{n-1}$ can be obtained by from $\mathcal{Y}_1$ by cyclicly permuting the indices $i\rightarrow i+1$ and $n\rightarrow 1$. We can find a path between these two algebras as follows. Define
\[
 \mathcal{Y}_l:=\{ Q_{[k]}\,| \,l+1\leq k \leq n-1\}\cup\{ Q_{[2\dots k]}\,| \,3\leq k \leq l+1\}.
\]
These are labelling Abelian algebra with $\mathcal{Y}_{l-1} \backslash \mathcal{Y}_{l}= Q_{[l]}$ and $\mathcal{Y}_{l}\backslash \mathcal{Y}_{l-1}= Q_{[2\dots l+1]}$. If we determine the overlap associated to each step we find the following connection coefficients:
\[
 B_{\vec s\vec k}=\prod_{l=2}^{n-1} B_{s_l k_l}(R_h^{1,[2..l],l+1},C^{1,[2..l],l+1}).
\] 
These are the multivariate Krawtchouk polynomials of Tratnik type. They depend on $n-2$ parameters $R_h^{1,[2..l],l+1}$. They are constructed in the same way as the multivariate Racah polynomials were constructed as connections coefficients between labelling Abelian algebras for the Racah algebra in \cite{DeBie&Genest$vandeVijver&Vinet}. The number of parameters follows from the number of steps we needed to move through the connection graph. By proposition \ref{connectiongraph} we know that we can find paths up to $\frac{(n-1)(n-2)}{2}$ steps. Further in the paper we will show how to get the multivariate Krawtchouk polynomials of Griffiths type depending on $\frac{(n-1)(n-2)}{2}$ parameters.

\subsection{$6j$- and $9j$-symbols}
The $6j$-, $9j$- and in general the $3nj$-symbols for the oscillator algebra $\mathfrak{h}$ can be cast into the framework presented in this article as they are specific overlaps between recoupled bases. Consider the $6j$-symbols. Given the algebra  $\mathfrak{h}\oplus \mathfrak{h}\oplus\mathfrak{h}$, the $6j$-symbols or Racah coefficients are the connection coefficients for coupling the first two oscillator algebras and the last two:
\[
(\mathfrak{h}\oplus \mathfrak{h})\oplus\mathfrak{h} \rightarrow \mathfrak{h}\oplus( \mathfrak{h}\oplus\mathfrak{h}).
\] 
In our framework this is equivalent with finding the overlap between the bases diagonalized by $Q_{12}$ and $Q_{23}$. The $6j$-symbols are therefore up to a normalization equal to $B_{sk}(R_h,C)$. One could ask where the $6$ $j$'s of the  $6j$-symbols are. The indices $s$ and $k$ are related to the eigenvalues of $Q_{12}$ and $Q_{23}$ so they are related to $j_{12}$ and $j_{23}$. The number $\mathcal{R}_h$ is independent of the representation used. The remaining four $j$'s are hidden in the Casimir $C$. According to formula (\ref{Casimir}) the Casimir depends on $Q_1$, $Q_2$, $Q_3$ and $Q_{123}$ giving the remaining $j_1$, $j_2$, $j_3$ and $j_{123}$.

A similar analysis can be given for the $9j$-symbols. The $9j$-symbols were already identified as bivariate Krawtchouk polynomials in \cite{Zhedanov}. They are obtained by considering the algebra $\mathfrak{h}\oplus \mathfrak{h}\oplus\mathfrak{h}\oplus\mathfrak{h}$ and the following two bases: the first basis obtained by coupling first factor with the second and the third with the fourth and the second basis defined by coupling the first factor with the third and the second with the fourth.  i.e.
\[
 (1\oplus2)\oplus (3\oplus 4) \rightarrow (1\oplus 3 )\oplus (2 \oplus 4)  .
\]
In our framework this is equivalent with finding the overlap coefficients between the bases diagonalized by $\langle Q_{12}, Q_{34}\rangle$ and $\langle Q_{13},Q_{24} \rangle$. We already calculated the connection coefficients in formula (\ref{overlapsl4}). We repeat the solution here:
\[
 B_{\vec s \vec k}=\sum_{l_2} B_{s_2l_2}(R_h^{13,2,4},C^{13,2,4})B_{s_1k_1}(R_h^{1,2,3},C^{1,2,3}) B_{l_2 k_2}(R_h^{4,3,12},C^{4,3,12}).
\]
These numbers are multivariate Krawtchouk polynomials depending on three parameter $R_h^{13,2,4}$,$R_h^{1,2,3}$  and $R_h^{4,3,12}$. These numbers $B_{\vec s \vec k}$ are up to normalization equal to the $9j$-symbols. The $9$ $j$'s are found in the following way: $\vec s$ and $\vec k$ are related to the eigenvalues of $\{Q_{12}, Q_{34}\}$ and  $\{Q_{13},Q_{24}\}$ respectively. The others are found by considering the three Casimir elements appearing in the formula. We can add the following generators: 
\[\{Q_{1},Q_{2},Q_{3},Q_{4},Q_{123},Q_{1234} \}.\]
 This gives a total of $10$ numbers. This is one too many. The summation runs over $l_1$ which appears not only as an index but also in $C^{1,2,3}$. This number which is related to the eigenvalues of $Q_{123}$ can be considered as being summed away. This leaves us with $9$ generators related to the $9$ $j$'s in the $9j$-symbols. A similar analysis can be done for any $3nj$-symbol.

\subsection{Automorphisms of $\mathcal{R}_n(\mathfrak{h})$ and $\mathfrak{sl}_{n-1}$}
For each labelling Abelian algebra of $\mathcal{R}_n(\mathfrak{h})$ we are able to construct a set of operators that generate $\mathfrak{sl}_{n-1}$. By Corollary \ref{isomorph} these sets of operators must generate the same algebra. This leads to automorphisms of $\mathfrak{sl}_{n-1}$. In this section we will give a few examples. Additionally we will give the group element of Lie group $\textup{SO}_{n-1}$ corresponding to each automorphism.

\begin{example}\label{example1}
Take the algebra $\mathcal{R}_3(\mathfrak{h})$ and construct $\mathfrak{sl}_2$ from the labelling Abelian algebra $\{Q_{12}\}$: As per formula (\ref{sl2A}):
\begin{align*}
 e&:=\lambda([Q_{12},[Q_{12},Q_{23}]]+a_{12}[Q_{12},Q_{23}]), \\
 f&:=\lambda([Q_{12},[Q_{12},Q_{23}]]-a_{12}[Q_{12},Q_{23}]), \\
 h&:=[e,f].
 \end{align*}
 Consider the permutation $(1\leftrightarrow 2)$. This can be represented by a twist of two branches of our tree as in Proposition \ref{connectiongraph}:
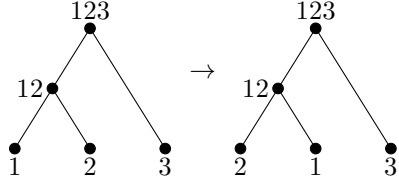
\begin{center}
\begin{tikzpicture}
\filldraw 	(0,0) circle (2pt) node[anchor=north] {$1$}
		(0.5,0.8) circle (2pt) node[anchor=east] {$12$}
		(1,1.6) circle (2pt)node[anchor=south]{$123$}
		(1,0) circle (2pt) node[anchor=north]{$2$}
		(2,0) circle (2pt) node[anchor=north]{$3$};
\draw (0,0)--(1,1.6);
\draw (0.5,0.8)--(1,0);
\draw (1,1.6)--(2,0);
\draw (2.5,1) node{ $\rightarrow$ };

\filldraw 	(3,0) circle (2pt) node[anchor=north] {$2$}
		(3.5,0.8) circle (2pt) node[anchor=east] {$12$}
		(4,1.6) circle (2pt)node[anchor=south]{$123$}
		(4,0) circle (2pt) node[anchor=north]{$1$}
		(5,0) circle (2pt) node[anchor=north]{$3$};
\draw (3,0)--(4,1.6);
\draw (3.5,0.8)--(4,0);
\draw (4,1.6)--(5,0);
		
\end{tikzpicture}
\end{center}
Keeping in mind that $[Q_{12},Q_{13}]=-[Q_{12},Q_{23}]$ by formula (\ref{Fijk}), the images of the generators under this permutation are $\tilde{h}=h$, $\tilde{e}=-e$ and $\tilde{f}=-f$. 
The automorphism of $\textup{SL}_2$ is constructed as follows. Let
 \begin{equation}
U  \begin{pmatrix}
     \frac{h}{2} & f \\
     e & -\frac{h}{2}
  \end{pmatrix}U^{-1}=
   \begin{pmatrix}
     \frac{\tilde{h}}{2} & \tilde{f} \\
     \tilde{e} & -\frac{\tilde{h}}{2}
  \end{pmatrix}.
 \end{equation}
We want to solve for $U$ with $U \in \textup{SL}_2$. Conjugation by $U$ is the related automorphism of $\textup{SL}_2$. In this case 
\[
U_{(12)} =\begin{pmatrix} 1 & 0 \\ 0 & -1 \end{pmatrix}.
\]
\end{example}
\begin{example}\label{example2} 
 Next we construct $\mathfrak{sl}_2$ from the labelling Abelian algebra $\{Q_{23}\}$. We are constructing an isomorphism corresponding to these two trees:
 
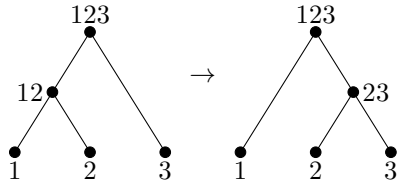
\begin{center}
\begin{tikzpicture}
\filldraw 	(0,0) circle (2pt) node[anchor=north] {$1$}
		(0.5,0.8) circle (2pt) node[anchor=east] {$12$}
		(1,1.6) circle (2pt)node[anchor=south]{$123$}
		(1,0) circle (2pt) node[anchor=north]{$2$}
		(2,0) circle (2pt) node[anchor=north]{$3$};
\draw (0,0)--(1,1.6);
\draw (0.5,0.8)--(1,0);
\draw (1,1.6)--(2,0);
\draw (2.5,1) node{ $\rightarrow$ };

\filldraw 	(3,0) circle (2pt) node[anchor=north] {$1$}
		(4.5,0.8) circle (2pt) node[anchor=west] {$23$}
		(4,1.6) circle (2pt)node[anchor=south]{$123$}
		(4,0) circle (2pt) node[anchor=north]{$2$}
		(5,0) circle (2pt) node[anchor=north]{$3$};
\draw (3,0)--(4,1.6);
\draw (4.5,0.8)--(4,0);
\draw (4,1.6)--(5,0);
		
\end{tikzpicture}
\end{center}
This is a swap as in Proposition \ref{connectiongraph}.
  We obtain the second set of generators by permuting the permutation $(1\rightarrow 2\rightarrow 3 \rightarrow 1)$. The transformed generators $\tilde{h}$, $\tilde{e}$ and $\tilde{f}$ are linear combinations of the original generators $h$, $e$ and $f$:
\begin{align*}
 \tilde{h}&=2\frac{\sqrt{a_1a_2a_3a_{123}}}{a_{12}a_{23}}(e+f)+\frac{a_1a_2+a_2 a_3+a_2^2-a_1a_3}{a_{12}a_{23}}h, \\
 \tilde{e}&=\frac{a_1a_3}{a_{12}a_{23}}e-\frac{a_2a_{123}}{a_{12}a_{23}}f-\frac{\sqrt{a_1a_2a_3a_{123}}}{a_{12}a_{23}}h,\\
 \tilde{f}&=-\frac{a_1a_3}{a_{12}a_{23}}e+\frac{a_2a_{123}}{a_{12}a_{23}}f-\frac{\sqrt{a_1a_2a_3a_{123}}}{a_{12}a_{23}}h .
\end{align*}
This set of equalities gives an automorphism of $\mathfrak{sl}_2$. We find the corresponding group element of $\textup{SO}_{2}$ by solving the following matrix equation:
 \begin{equation}
U  \begin{pmatrix}
     \frac{h}{2} & f \\
     e & -\frac{h}{2}
  \end{pmatrix}U^{-1}=
   \begin{pmatrix}
     \frac{\tilde{h}}{2} & \tilde{f} \\
     \tilde{e} & -\frac{\tilde{h}}{2}
  \end{pmatrix}.
 \end{equation}
Solving for $U$ gives
\[
U_{(123)} =\begin{pmatrix}  \sqrt{\frac{a_1a_3}{a_{12}a_{23}}} &   -\sqrt{\frac{a_2a_{123}}{a_{12}a_{23}}} \\  \sqrt{\frac{a_2a_{123}}{a_{12}a_{23}}}&  \sqrt{\frac{a_1a_3}{a_{12}a_{23}}} \end{pmatrix}.
\]
A straightforward calculation shows that $U_{(123)}$ is indeed an orthogonal matrix. It represents a rotation with angle $\theta_{1,2,3}$ defined by
\[
 \cos(\theta_{1,2,3})= \sqrt{\frac{a_1a_3}{a_{12}a_{23}}}.
\]

\end{example}

\begin{example}\label{example3}
Consider in the algebra $\mathcal{R}_4(\mathfrak{h})$ the Lie algebra $\mathfrak{sl}_3$ constructed from the labelling Abelian algebra $\{Q_{12},Q_{123}\}$.
\begin{align*}
 e_{12}&:=\lambda_1([Q_{12},[Q_{12},Q_{23}]]+a_{12}[Q_{12},Q_{23}]), \\
 e_{21}&:=\lambda_1([Q_{12},[Q_{12},Q_{23}]]-a_{12}[Q_{12},Q_{23}]), \\
 h_1&:=[e_{12},e_{21}],\\
  e_{23}&:=\lambda_2([Q_{123},[Q_{123},Q_{34}]]+a_{123}[Q_{123},Q_{34}]), \\
 e_{32}&:=\lambda_2([Q_{123},[Q_{123},Q_{34}]]-a_{123}[Q_{123},Q_{34}]),\\
 h_2&:=[e_{23},e_{32}], \\
 e_{13}&:=[e_{12},e_{23}], \\
 e_{31}&:=[e_{32},e_{21}].
\end{align*}
When performing the permutation  $(1\rightarrow 2\rightarrow 3 \rightarrow 1)$ we find new generators. These generators are eigenvectors of the labelling Abelian algebra $\{Q_{23},Q_{123} \}$. We will not explicitly express the new generators as linear combinations but we will give $U \in \textup{SL}_3$. To find $U$ we solve the following equation
\begin{equation}\label{SL3}
U 	\begin{pmatrix}	
		 \frac{2h_1+h_2}{3} & e_{21} &e_{31} \\
		e_{12} & \frac{-h_1+h_2}{3} & e_{32} \\
		e_{13} & e_{23} & \frac{-h_1-2h_2}{3}
	\end{pmatrix}U^{-1} =
	\begin{pmatrix}	
		 \frac{2\tilde{h}_1+\tilde{h}_2}{3} & \tilde{e}_{21} &\tilde{e}_{31} \\
		\tilde{e}_{12} & \frac{-\tilde{h}_1+\tilde{h}_2}{3} & \tilde{e}_{32} \\
		\tilde{e}_{13} & \tilde{e}_{23} & \frac{-\tilde{h}_1-2\tilde{h}_2}{3}
	\end{pmatrix}.
\end{equation}
For the given permutation  we find
\[
 U_{(123)}=
\begin{pmatrix}
   \sqrt{\frac{a_1a_3}{a_{12}a_{23}}} &   -\sqrt{\frac{a_2a_{123}}{a_{12}a_{23}}} & 0\\ 
    \sqrt{\frac{a_2a_{123}}{a_{12}a_{23}}}&  \sqrt{\frac{a_1a_3}{a_{12}a_{23}}} & 0 \\
    0 &0 & 1
 \end{pmatrix}.
\]
The $2\times 2$ in the upper right corner is exactly $U_{(123)}$ in the $n=2$ case. If  $U_{(123)}$ acts on a three dimensional space it represents a planar rotation over an angle $\theta_{1,2,3}$. We will denote this matrix alternatively by $R_{x_1x_2}(\theta_{1,2,3}):=U_{(123)}$. The $x_1x_2$ index represents the plane that is being rotated. 
This overlap can be represented by a swap as in Proposition \ref{connectiongraph}:
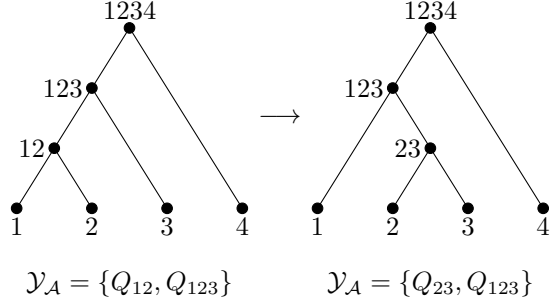
\begin{center}
\begin{tikzpicture}
\filldraw 	(0,0) circle (2pt) node[anchor=north] {1}
		(0.5,0.8) circle (2pt) node[anchor=east] {12}
		(1,1.6) circle (2pt)node[anchor=east]{123}
		(1,0) circle (2pt) node[anchor=north]{2}
		(2,0) circle (2pt) node[anchor=north]{3}
		(1.5,2.4) circle (2pt) node[anchor=south] {1234}
		(3,0) circle (2pt) node[anchor=north]{4};
\draw (0,0)--(1.5,2.4);
\draw (0.5,0.8)--(1,0);
\draw (1,1.6)--(2,0);
\draw (1.5,2.4)--(3,0);
\draw (1.5,-1)node {$\mathcal{Y}_{\mathcal{A}}=\{Q_{12},Q_{123}\}$};

\draw(3.5,1.2) node{$\longrightarrow$};

\filldraw 	(4,0) circle (2pt) node[anchor=north] {1}
		(5.5,0.8) circle (2pt) node[anchor=east] {23}
		(5,1.6) circle (2pt)node[anchor=east]{123}
		(5,0) circle (2pt) node[anchor=north]{2}
		(6,0) circle (2pt) node[anchor=north]{3}
		(5.5,2.4) circle (2pt) node[anchor=south] {1234}
		(7,0) circle (2pt) node[anchor=north]{4};
\draw (4,0)--(5.5,2.4);
\draw (5.5,0.8)--(5,0);
\draw (5,1.6)--(6,0);
\draw (5.5,2.4)--(7,0);
\draw (5.5,-1)node {$\mathcal{Y}_{\mathcal{A}}=\{Q_{23},Q_{123}\}$};
\end{tikzpicture}
\end{center}
\end{example}

\begin{example}\label{example4}
Consider in the algebra $\mathcal{R}_4(\mathfrak{h})$ again the Lie algebra $\mathfrak{sl}_3$ constructed from the labelling Abelian algebra $\mathcal{Y}_1:=\{Q_{12},Q_{123}\}$.  The second labelling Abelian algebra we consider is $\mathcal{Y}_2:=\{ Q_{12},Q_{34} \}$. As before we can represent this overlap by a swap as in Proposition \ref{connectiongraph}. 
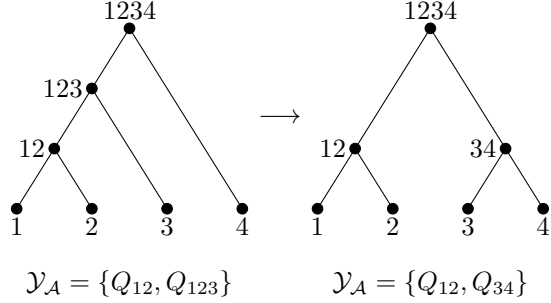
\begin{center}
\begin{tikzpicture}
\filldraw 	(0,0) circle (2pt) node[anchor=north] {1}
		(0.5,0.8) circle (2pt) node[anchor=east] {12}
		(1,1.6) circle (2pt)node[anchor=east]{123}
		(1,0) circle (2pt) node[anchor=north]{2}
		(2,0) circle (2pt) node[anchor=north]{3}
		(1.5,2.4) circle (2pt) node[anchor=south] {1234}
		(3,0) circle (2pt) node[anchor=north]{4};
\draw (0,0)--(1.5,2.4);
\draw (0.5,0.8)--(1,0);
\draw (1,1.6)--(2,0);
\draw (1.5,2.4)--(3,0);
\draw (1.5,-1)node {$\mathcal{Y}_{\mathcal{A}}=\{Q_{12},Q_{123}\}$};

\draw(3.5,1.2) node{$\longrightarrow$};

\filldraw 	(4,0) circle (2pt) node[anchor=north] {1}
		(4.5,0.8) circle (2pt) node[anchor=east] {12}
		(6.5,0.8) circle (2pt)node[anchor=east]{34}
		(5,0) circle (2pt) node[anchor=north]{2}
		(6,0) circle (2pt) node[anchor=north]{3}
		(5.5,2.4) circle (2pt) node[anchor=south] {1234}
		(7,0) circle (2pt) node[anchor=north]{4};
\draw (4,0)--(5.5,2.4);
\draw (4.5,0.8)--(5,0);
\draw (6.5,0.8)--(6,0);
\draw (5.5,2.4)--(7,0);
\draw (5.5,-1)node {$\mathcal{Y}_{\mathcal{A}}=\{Q_{12},Q_{34}\}$};
\end{tikzpicture}
\end{center}

The generators corresponding to the second Abelian algebra are:
\begin{align*}
 \tilde{e}_{12}&:=\lambda_{12}([Q_{12},[Q_{12},Q_{234}]]+a_{12}[Q_{12},Q_{234}]), \\
  \tilde{e}_{21}&:=\lambda_{12}([Q_{12},[Q_{12},Q_{234}]]-a_{12}[Q_{12},Q_{234}]), \\
 \tilde{h}_1&:=[e_{12},e_{21}],\\
 \tilde{e}_{23}&:=\lambda_{34}([Q_{34},[Q_{34},Q_{124}]]+a_{34}[Q_{34},Q_{124}]), \\
  \tilde{e}_{32}&:=\lambda_{34}([Q_{34},[Q_{34},Q_{124}]]-a_{34}[Q_{34},Q_{124}]),\\
 \tilde{h}_2&:=[e_{23},e_{32}] ,\\
  \tilde{e}_{13}&:=[e_{12},e_{23}] ,\\
  \tilde{e}_{31}&:=[e_{32},e_{21}].
\end{align*}
Solving equation \eqref{SL3} results in

\[
U_{\mathcal{Y}_1\mathcal{Y}_2}:= \begin{pmatrix}
     1 &0 & 0 \\
     0 & \sqrt{\frac{a_3a_{1234}}{a_{123}a_{34}}} & -\sqrt{ \frac{a_4a_{12}}{a_{123}a_{34}}} \\
     0 &  \sqrt{ \frac{a_4a_{12}}{a_{123}a_{34}}} & \sqrt{\frac{a_3a_{1234}}{a_{123}a_{34}}}
 \end{pmatrix}.
\]
This is again an orthogonal matrix. Similarly to the previous example one can see that the lower-right $2\times 2$ matrix is a rotation matrix. The matrix $U_{\mathcal{Y}_1\mathcal{Y}_2}$ represents a planar rotation matrtx over an angle $\pi/2-\theta_{12,3,4}$ but in a different plane. We denote this matrix by
\[
 R_{x_2x_3}(\theta_{12,3,4}-\pi/2).
\]
\end{example}

Given two labelling Abelian algebras, we are able to construct an isomorphism of $\mathfrak{sl}_{3}$ and its corresponding rotation matrix. The previous examples show that if one labelling Abelian algebra is obtained by a swap on the other, the resulting rotation matrices are planar. Twists on the other hand give reflections.  This leads to an alternative way to construct the rotation matrix of an isomorphism related to two labelling Abelian algebras. We choose a path in the connection graph between the vertices related to the labelling Abelian algebras. Every step along the path in the connection graph can be represented by a swap as in Proposition \ref{connectiongraph}. Each of these swaps leads to a planar rotation. The final rotation matrix will be the product of each planar rotation found along the path.
For example, consider these two labelling Abelian algebras:  $\langle Q_{12},Q_{34} \rangle$ and $\langle Q_{13},Q_{24} \rangle$. The intermediate bases are the following:
 \[
 \langle Q_{12},Q_{34} \rangle \rightarrow \langle Q_{12},Q_{123} \rangle \rightarrow \langle Q_{13},Q_{123} \rangle \rightarrow \langle Q_{13},Q_{24} \rangle.
 \]
For each step we construct the corresponding rotation matrix. This results in 
\[
R:=  R_{x_2x_3}(\pi/2-\theta_{12,3,4})R_{x_1x_2}(-\theta_{1,2,3})R_{x_2x_3}(\theta_{13,2,4}-\pi/2).
\]
It is known that every rotation matrix $R \in \textup{SO}_{3}$ can be written as the following product of planar rotations:
\[
 R_{x_2x_3}(\theta_1)R_{x_1x_2}(\theta_2)R_{x_2x_3}(\theta_3).
\]
The angles $\theta_1$, $\theta_2$ and $\theta_3$ are the so-called Euler angles. By choosing the right path through the connection graph it is possible to give the decomposition of the rotation matrix in planar matrices.
The Euler angles also show up in the in the overlap coefficients between the given bases. The overlap coefficients \eqref{overlapsl4} are multivariate Krawtchouk polynomials of Griffiths type depending on three parameters $R_h^{13,2,4}$, $R_h^{1,2,3}$  and $R_h^{4,3,12}$. The relationship between these parameters and the angles is the following:
\[
 R_{h}^{K,L,M}=\cos(2\theta_{K,L,M}).
\]
This relation between multivariate Krawtchouk polynomials and rotations had been discussed earlier in  \cite{GVZ5}.

We generalize the previous analysis to any $n$.  To any pair of labelling Abelian algebras we are able to construct overlap coefficients and an isomorphism of  $\mathfrak{sl}_{n-1}$ represented by a rotation matrix in $\textup{SO}_{n-1}$. Any rotation matrix in $\textup{SO}_{n-1}$ can be written as the product of $\frac{(n-1)(n-2)}{2}$ planar rotations. In the same way the overlap coefficients are multivariable Krawtchouk polynomials of Griffiths type depending on $\frac{(n-1)(n-2)}{2}$ parameters. By Proposition \ref{connectiongraph} it takes at most $\frac{(n-1)(n-2)}{2}$ steps through the connection graph. Each step provides us with a parameter and a planar rotation matrix leading to the right number of parameters and planar rotations. We wil showcase this with an example that also shows the link with Krawtchouk polynomials of Tratnik type:
\begin{example}
Consider the labelling Abelian algebras 
\[
  \mathcal{Y}_{initial}:=\{ Q_{[k]}\,|\, 1<k<n\} , \quad \mathcal{Y}_{final}:=\{ Q_{[k\dots n]}\,|\, 1< k<n\}.
\]
Between these two labelling Abelian algebras, it will take $\frac{(n-1)(n-2)}{2}$ steps to find the overlap. \newline
\underline{Step $1$:} Replace $Q_{12}$ by $Q_{23}$ in $\mathcal{Y}_{initial}$. The overlap coefficients are univariate Krawtchouk polynomials:
\[
B_{s^{1}_1 s^{2}_1}(R_h^{1,2,3},C^{1,2,3}).
\]
The vector index $ s^{1}$ labels the basis diagonalized by $\mathcal{Y}_{initial}$. The index $s^{2}$ labels the basis diagonalized by the new labelling Abelian algebra.  The related rotation matrix is given by
\[
R_{x_1x_2}(\theta_{1,2,3}).
\] 
\underline{Step $2$:} We perform two steps: Replace $Q_{123}$ by $Q_{234}$ and then $Q_{23}$ by $Q_{34}$. The overlap coefficients are bivariate Krawtchouk polynomials of Tratnik type:
\[
  B_{s^{2}_1 s^{3}_1}(R_h^{1,23,4},C^{1,23,4})B_{s^{2}_2 s^{3}_2}(R_h^{2,3,4},C^{2,3,4}).
\]
The related rotation matrix is given by 
\[
 R_{x_1x_2}(\theta_{2,3,4})R_{x_2x_3}(\theta_{1,23,4}).
\]
\underline{Step $k-1$:} Consider the first $k-1$ elements in the labelling Abelian algebra: 
\[
	\{Q_{k-1k},Q_{[k-2.. k]}, \dots, Q_{[2.. k]}, Q_{[k]} \}.
\]
We perform $k-1$ steps. Replace $Q_{[k]}$ by $Q_{[2.. k+1]}$. Next replace $Q_{[2 .. k]}$ by $Q_{[3 ..  k+1]}$ and so on until $Q_{k-1k}$ is replaced by $Q_{kk+1}$. These $k-1$ steps lead to overlap coefficients that are a product of $k-1$ univariate Krawtchouk polynomials. These are more specifically multivariate Krawtchouck polynomials of Tratnik type:
\[
	\prod_{l=1}^{k-1} B_{s^{k-1}_{l} s^{k}_{l}}(R_h^{l,[l+1..k],k+1},C^{l,[l+1..k],k+1}).
\]
The related rotation matrix is given by
\[
 \prod_{l=1}^{k-1} R_{x_{k-l}x_{k+1-l}}(\theta_{l,[l+1..k],k+1}).
\]
Combining all $n-2$ steps we come to the conclusion that the overlap coefficients are given by
\[
B_{\vec s^1 \vec s^{n-1}}= \sum_{m=2} ^{n-2}\sum_{s_{l}^{m}} \prod_{k=2}^{n-1} \prod_{l=1}^{k-1} B_{s^{k-1}_{l} s^{k}_{l}}(R_h^{l,[l+1..k],k+1},C^{l,[l+1..k],k+1}).
\]  
The overlap coefficients are multivariate Krawtchouk polynomials of Griffiths type which are themselves a sum of products of Krawtchouk polynomials of Tratnik type. The overlap coefficients depend on $\frac{(n-1)(n-2)}{2}$ parameters of the form $R_h^{l,[l+1..k],k+1}$.
The corresponding rotation matrix is given by a product of $\frac{(n-1)(n-2)}{2}$ rotation matrices.
\[
 \prod_{k=2}^{n-1}\prod_{l=1}^{k-1} R_{x_{k+1-l}x_{k-l}}(\theta_{l,[l+1..k],k+1})
\]

\end{example}

 \section{Conclusion}
We introduced the oscillator Racah algebra $\mathcal{R}_n(\mathfrak{h})$. We have shown how to embed $\mathfrak{sl}_{n-1}$ into the algebra $\mathcal{R}_n(\mathfrak{h})$. This connects the representation theory for both algebras. In finite irreducible representations we considered bases diagonalized by labelling Abelian algebras of $\mathcal{R}_n(\mathfrak{h})$. The overlap coefficients between a pair of bases are shown to be multivariate Krawtchouk polynomials of Tratnik or Griffiths type. Isomorphisms of $\mathfrak{sl}_{n-1}$ related to pairs of labelling Abelian algebras and their corresponding Lie group elements were constructed and their link to the overlap coefficients is explained. This has provided an explanation as to why the recoupling coefficients of the oscillator algebra and the matrix elements of the restrictions to $\textup{O}(n-1)$ of symmetric representations of $\textup{SL}(n-1)$ are generically given in terms of the multivariate Krawtchouk polynomials of Griffiths.

\section{Acknowledgements}
NC is partially supported by Agence National de la Recherche Projet AHA ANR-18-CE40-0001 and is gratefully holding a CRM-Simons professorship. WVDV thanks the Fonds Professor Frans Wuytack for supporting his research. He is also grateful for the hospitality offered by him at the CRM during his stay. The research of LV is supported in part  by a discovery grant of the Natural Science and Engineering Research Council (NSERC) of Canada.

\appendix

\section{Calculation of overlap coefficients}\label{Overlap}
Let $V$ be a finite dimensional representation of $\mathfrak{sl}_2$ and $\tilde{.}$ an automorphism of $\mathfrak{sl}_2$. The element $h$ is a Cartan generator of $\mathfrak{sl}_2$. Let $\{ \psi_k \}$ be  an eigenbasis of $h$ and $\{ \phi_s \}$ be an eigenbasis for $\tilde{h}$. The indices $k$ and $s$ run form $0$ to $N$ with $\dim(V)=N+1$. We are interested in the overlap coefficients $B_{ks}$ between these bases:
\begin{equation}\label{overlapnew}
 \phi_s=\sum_{k=0}^N B_{sk}\psi_k.
\end{equation}

 \[
    h\psi_k=\mu_k\psi_k \qquad \tilde{h}\phi_s=\nu_s\phi_s.
 \]
 The algebra $\mathfrak{sl}_2$ has algebra relations $[h,e]=2e$ and $[h,f]=-2f$ with $e$ is the raising operator and $f$ the lowering operator on $\{ \psi_k \}$:
 \[
  e\psi_k=e_{kk+1}\psi_{k+1} \quad f \psi_k=f_{kk-1} \psi_{k-1}
 \]
 and $\mu_k=\mu_0+2k$. From the algebra relation $[e,f]=h$ it follows that
 \[
f_{kk-1}e_{k-1k}- e_{kk+1}f_{k+1k}=\mu_k.
\]
Let $A_k:=e_{kk-1}f_{k-1k}$. Then we have
\[
 A_{k}-A_{k+1}=2k+\mu_0.
\]
From this we find
\[
A_k=-k(k-1)-\mu_0 k-\Omega
\]
with $\Omega \in \mathbb{R}$. We express $\tilde{h}$ as a linear combination of $h$, $e$ and $f$. 
\[
  \tilde{h}=R_h h +R_e e+ R_f f
\]
with $R_eR_f+R_h^2=1$. We have set up everything we need to find the overlap coefficients. Let the operator $\tilde{h}$ act on both sides of equality (\ref{overlapnew}).
\[
 \tilde{h}\psi_s=\sum_{k=0}^N B_{sk}(R_h h +R_e e+ R_f f) \psi_k.
\]
This gives
\[
\nu_s\phi_s=\sum_{k=0}^N B_{sk}(R_h\mu_k\psi_k+R_e e_{kk+1}\psi_{k+1}+R_f f_{kk-1}\psi_{k-1}).
\]
We expand the left hand side into the basis $\psi_k$ and we gather the terms on the right hand side:
\[
\sum_{k=0}^N \nu_s B_{sk} \psi_k=\sum_{k=0}^N (B_{sk}R_h\mu_k+B_{sk-1}R_e e_{k-1k}+B_{sk+1}R_f f_{k+1k})\psi_k.
\]
From this we find the recurrence relation
\[
\nu_s B_{sk}=B_{sk+1}R_f f_{k+1k}+ B_{sk}R_h\mu_k+B_{sk-1}R_e e_{k-1k}.
\]
We want to recognize this recurrence relation as one of the family of orthogonal polynomials. Let
\[
\tilde{B}_{ks}=\left(\prod_{t=2}^k f_{tt-1}R_f\right)B_{sk}
\]
to find
\[
 \nu_s \tilde{B}_{sk}=\tilde{B}_{sk+1}+R_h\mu_k\tilde{B}_{sk}+R_eR_f e_{k-1k}f_{k-1k}\tilde{B}_{sk-1}.
\]
We write the coefficients as polynomials in $x=\nu_s$:
\begin{equation}\label{eq1}
 x\tilde{B}_k(x)=\tilde{B}_{k+1}(x)+R_h\mu_k\tilde{B}_k(x)+R_eR_fA_k\tilde{B}_{k-1}(x).
\end{equation}
We want to compare this with the recurrence relation of the normalized Krawtchouk polynomials as defined in \cite{Koekoek&Lesky&Swarttouw-2010} :
\[
 xp_n(x)=p_{n+1}+(n(1-2r)+rN)p_n(x)+r(1-r)n(N+1-n)p_{n-1}(x)
\]
with $n=0,1, \dots , N$. Let $x=\alpha y+\beta$ and introduce $q_n(y)=p_n(\alpha y+\beta)/\alpha^n$. The polynomial $q_n(x)$ satisfies the following recurrence relation:
\[
yq_n(y)=q_{n+1}+\frac{n(1-2r)+rN-\beta}{\alpha}q_n(y)+\frac{r(1-r)n(N+1-n)}{\alpha^2}q_{n-1}(y).
\]
We retrieve equation (\ref{eq1}) if we set
\begin{equation*}
\alpha=\frac{1}{2}, \quad r=\frac{1-R_h}{2}, \quad \beta=\frac{N}{2}, \quad k=n, \quad \Omega=0, \quad \mu_0=-N.
\end{equation*}
We explicitly write down the polynomials $\tilde{B}_k(x)$.
\begin{align*}
 \tilde{B}_k(x)&=2^kp_k\left(\frac{x+N}{2}\right)\\
 			&=(-N)_k (1-R_h)^k K_k\left(\frac{x+N}{2}; \frac{1-R_h}{2},N\right) \\
 			&=(-N)_k (1-R_h)^k {}_2F_1\left(\substack{-k,-\frac{x+N}{2}\\-N}\large| \frac{2}{1-R_h}\right).
\end{align*}
The overlap coefficients are the Krawtchouk polynomials $K_k(x)$ (defined in the last line of the equation above) up to a normalization factor. 
\[
 B_{sk}=\left(\prod_{t=2}^k f_{tt-1}R_f\right)(-N)_k (1-R_h)^k K_k\left(\frac{\nu_s+N}{2}; \frac{1-R_h}{2},N\right).
\]

\end{document}